\documentclass[a4paper,reqno, 11pt]{amsart}

\usepackage[margin=3cm]{geometry}
\usepackage{amssymb}
\usepackage{amsmath}
\usepackage{amsthm}
\usepackage{mathrsfs}
\usepackage{tikz}
\usepackage{tikz-cd}
\usepackage{enumitem}

\usepackage[
    backend=biber,
    style=numeric
]{biblatex}
\usepackage{hyperref}

\newtheorem{theorem}{Theorem}[section]
\newtheorem{theorem*}{Theorem}
\newtheorem{lemma}[theorem]{Lemma}
\newtheorem{corollary}[theorem]{Corollary}

\newtheorem{proposition}[theorem]{Proposition}
\theoremstyle{definition}
\newtheorem{definition}[theorem]{Definition}
\newtheorem{example}[theorem]{Example}

\newcommand{\CC}{\mathbb{C}}

\newcommand{\TT}{\mathbb{T}}

\newcommand{\Nn}{\mathcal{N}}

\newcommand{\Cc}{\mathcal{C}}

\newcommand{\Ss}{\mathcal{S}}

\newcommand{\Gg}{\mathcal{G}}
\newcommand{\bB}{\mathfrak{B}}

\newcommand{\mM}{\mathfrak{M}}

\newcommand{\sspan}{\operatorname{span}}

\newcommand{\Aut}{\operatorname{Aut}}
\newcommand{\eps}{\varepsilon}

\newcommand{\Id}{\operatorname{id}}

\newcommand{\supp}{\operatorname{supp}}
\newcommand{\ess}{\mathrm{ess}}
\newcommand{\red}{\mathrm{r}}
\newcommand{\alg}{\mathrm{alg}}

\newcommand{\Mloc}{M_{\mathrm{loc}}}

\newcommand{\dom}{\mathrm{dom}}
\newcommand{\ran}{\mathrm{ran}}

\DeclareMathOperator{\clspan}{\overline{span}}

\theoremstyle{remark}
\newtheorem{remark}[theorem]{Remark}

\title{Essential Cartan subalgebras of $C^*$-algebras}

\author{Jonathan Taylor}
\address{Institut für Mathematik, University of Potsdam, Campus Golm, Haus 9, Karl-Liebknecht-Str. 24-25, 14476, Germany}
\email{jonathan.taylor@uni-potsdam.de}

\thanks{The author was supported by a stipend from the German Academic Exchange Service (DAAD) Funding Program 57450037 and was an affiliated member of the RTG2491 ``Fourier Analysis and Spectral Theory''}

\begin{document}
	
	\begin{abstract}
		We define essential Cartan pairs of $C^*$-algebras generalising the definition of Renault \cite{Renault_Cartan} and show that such pairs are given by essential twisted groupoid $C^*$-algebras as defined by Kwa\'sniewski and Meyer \cite{KwasMeyer_EssCrossProd}.
		We show that the underlying twisted groupoid is effective, and is unique up to isomorphism among twists over effective groupoids giving rise to the essential Cartan pair.
		We also show that for twists over effective groupoids giving rise to such pairs, the automorphism group of the twist is isomorphic to the automorphism group of the induced essential Cartan pair via explicit constructions.
	\end{abstract}
	
	\maketitle
	
	\section{Introduction}
	
	The class of $C^*$-algebras arising from groupoid constructions contains many important examples, including all commutative $C^*$-algebras and all group $C^*$-algebras.
	Renault's reconstruction theorem provides a valuable strategy for describing a \(C^*\)-algebra as the \(C^*\)-algebra of a twist over an \'etale groupoid.
	Renault \cite{Renault_Cartan} defined a \emph{Cartan subalgebra} of a \(C^*\)-algebra as a regular, non-degenerate, maximal commutative subalgebra equipped with a faithful conditional expectation.
	Renault showed that if a separable \(C^*\)-algebra has a Cartan subalgebra then it is isomorphic to a twisted groupoid $C^*$-algebra of a twist over a second-countable \'etale effective locally compact Hausdorff groupoid, via an isomorphism mapping the Cartan subalgebra to the algebra of continuous functions on the unit space of the groupoid.
	Moreover, this groupoid and twist are unique up to isomorphism among such groupoids.
	This reconstruction was extended by Raad \cite{Raad_GeneraliseRenault} and Kwa\'sniewski and Meyer \cite{KwasMeyer_NCCart} to Cartan subalgebras of non-separable \(C^*\)-algebras; the resulting groupoid and twist are no longer second-countable.
	
	One notable application of Renault's reconstruction theorem was by Li \cite{Li_ClassifCartan}, where it was shown that the inductive limit of a system of \(C^*\)-algebras with Cartan subalgebras also has a Cartan subalgebra, provided the connecting maps preserve the Cartan structure.
	A significant consequence of this is that every Elliott-classifiable $C^*$-algebra has a Cartan subalgebra, and so every classifiable $C^*$-algebra is a twisted groupoid $C^*$-algebra for a twist over a second-countable \'etale effective locally compact Hausdorff groupoid.
	This deeply intertwines the research of classification with the field of groupoids and their algebras, and allows for more geometric approaches to classification problems.
	
	Over time, the study of groupoid algebras widened to include non-Hausdorff groupoids.
	Exel and Pitts \cite{ExelPitts_AlgNonHDGrpds} define the \emph{essential} \(C^*\)-algebra of groupoids with sufficiently small isotropy, and provide an analogue of Renault's reconstruction theorem for inclusions of \(C^*\)-algebras with favourable state-extension properties.
	Alternatively, Kwa\'sniewski and Meyer \cite{KwasMeyer_EssCrossProd} give a suitable definition of essential groupoid \(C^*\)-algebra generalising the definition of Exel and Pitts in many cases, which they use to study inclusions of noncommutative subalgebras and their properties.
	The study of non-Hausdorff groupoids is in part motivated by their frequent appearance in the study of dynamics on topological spaces.
	Simple examples of Exel \cite{Exel_NonHausdorffGrpds} show that even group actions on topological spaces may have non-Hausdorff groupoids of germs.
	More generally, given a Hausdorff \'etale groupoid, one may wish to study the effective quotient of the groupoid to eliminate as much isotropy as possible while remaining \'etale.
	The resulting quotient is again Hausdorff if and only if the interior of the isotropy of the groupoid happens to be closed, which is difficult to determine.
	
	Kwa\'sniewski and Meyer employed many techniques surrounding the use of \emph{pseudoexpectations} of the inclusions they studied; these are a generalisation of conditional expectations that may take values in Hamana's injective hull (see \cite{Hamana_InjEnvs}) of the subalgebra, rather than the subalgebra itself.
	The main advantage of considering such expectations is that they always exist, whereas (genuine) conditional expectations must be either assumed or shown to exist.
	When the subalgebra considered is commutative, Hamana's injective hull coincides with the \emph{local multiplier algebra} (see \cite{FrankPaulsen_InjEnvCStarAlgs}), and so it suffices to consider \emph{local conditional expectations} which take values in the local multiplier algebra \(\Mloc(A)\) of the subalgebra \(A\).
	
	The main objects of study in this article are \emph{essential Cartan pairs}, which are a generalisation of Renault's Cartan pairs by allowing the consideration of pseudoexpectations in place of the conditional expectations in Renault's work.
	Equivalently, if a regular maximal commutative subalgebra detects ideals in the larger superalgebra (i.e. it is a \emph{virtual Cartan subalgebra} in the sense of Pitts \cite{Pitts_StructRegInc1}) then it will have a faithful pseudoexpectation and hence induce an essential Cartan pair.
	The main result of this article is a generalisation of Renault's reconstruction theorem for essential Cartan pairs, where the shift from genuine conditional expectations to pseudoexpectations manifests in the underlying reconstructed groupoids no longer necessarily being Hausdorff.
	
	\begin{theorem*}[{see Theorem~\ref{thm-isoEssTwstGrpdCstarAlg}}]\label{intro-thm1}
		Let $A\subseteq B$ be an essential Cartan pair with faithful local conditional expectation $E:B\to\Mloc(A)$.
		Let $(G(A,B),\Sigma(A,B))$ by the Weyl twist associated to $A\subseteq B$.
		There is a ${}^*$-isomorphism $C^*_\ess(G(A,B),\Sigma(A,B))\cong B$ that restricts to an isomorphism $C_0(G(A,B)^{(0)})\cong A$ and intertwines $E$ with the canonical local expectation $C^*_\ess(G(A,B),\Sigma(A,B))\to\Mloc(C_0(G(A,B)^{(0)}))$, where $\Mloc(A)$ is identified with $\Mloc(C_0(G(A,B)^{(0)}))$.
	\end{theorem*}

	In his proof of the reconstruction theorem for Cartan pairs, Renault introduces the \emph{evaluation map}, which sends elements of the \(C^*\)-algebra with a Cartan subalgebra to functions on the constructed groupoid twist.
	This, when combined with the \(j\)-map allowing the description of the reduced \(C^*\)-algebra of a twist as an algebra of functions on the twist, gives the desired isomorphism.
	This tool is no longer available for two (intrinsically related) reasons.
	One, there is no \(j\)-map for the essential groupoid \(C^*\)-algebra: at best its elements may be described as equivalence classes of functions but not as functions on the nose.
	Two, the evaluation map describes functions on twists by applying the conditional expectation and evaluating the resulting functions in the commutative subalgebra at points.
	Since we consider pseudoexpectations in place of genuine conditional expectations, we may not identify the values of our expectations with functions on the spectrum of the essential Cartan subalgebra.
	In fact, the local multiplier algebra of a commutative \(C^*\)-algebra \(C_0(X)\) can be identified with the quotient of the bounded Borel-measureable functions by the functions with meagre support (see \cite{Dixmier_espacesStone},\cite{Gonshor_InjHulls2}), and hence elements of \(\Mloc(C_0(X))\) are equivalence classes of functions on \(X\).
	Under this identification, it does not make sense to evaluate an element of \(\Mloc(C_0(X))\) at a single point of \(X\) (except in the rare case that the singleton of this point is open).

	Bice \cite{Bice_Sections} provided an alternative construction of the reduced and essential \(C^*\)-algebras of a twist over an \'etale groupoid as closures of a canonical subalgebra in a large ambient space of sections.
	This framework describes the \(j\)-map implicitly as an inclusion of spaces, and we will make use of this characterisation of the reduced and essential \(C^*\)-algebras throughout the article. 
	
	We also employ techniques used in the study of inclusions of \(C^*\)-algebras generated by inverse semigroup actions by Kwa\'sniewski and Meyer in \cite{KwasMeyer_EssCrossProd} and \cite{KwasMeyer_NCCart}.
	The generalisation to consider $C^*$-inclusions with local multiplier algebra valued conditional expectations arose in \cite{KwasMeyer_EssCrossProd} for defining the essential crossed product of inverse semigroup actions.
	For such crossed products, there is always a canonical faithful conditional expectation taking values in the local multiplier algebra of the subalgebra.
	Kwa\'sniewski and Meyer also introduce the notion of \emph{aperiodic} inclusions and \emph{aperiodic} actions, and prove a number of results for such inclusions in \cite{KwasMeyer_AperTopoFree}, \cite{KwasMeyer_EssCrossProd}, and \cite{KwasMeyer_PseudoExp}.
	We show that aperiodicity is automatically satisfied for masa inclusions, and so we may utilise these results.
	Combined with some point-set topological arguments to manage the non-Hausdorffness of the underlying groupoids, we are able to construct the isomorphism in Theorem~\ref{intro-thm1} by choosing canonical representatives of specific elements of the local multiplier algebra of an essential Cartan subalgebra.
	This leads to a homomorphism of dense subalgebras, which extends to the necessary completions and quotients by using the conditional expectation.
	
	We also show that if a twist over an \'etale groupoid with locally compact Hausdorff unit space gives rise to an essential Cartan pair then the associated Weyl twist is a twist over the effective quotient of this groupoid.
	Hence we show that twists over such groupoids descend to their effective quotients.
	Moreover, we show that any twist over an effective \'etale groupoid with locally compact Hausdorff unit space gives rise to an essential Cartan pair.
	
	More general regular masa inclusions also fall partly into this framework.
	Since such inclusions always have a unique local multiplier algebra-valued conditional expectation (see Corollary~\ref{cor-masaIncIsAper} and \cite[Theorem~3.6]{KwasMeyer_PseudoExp}), the only remaining condition to check is that the conditional expectation is faithful.
	If it is not, work of Kwa\'sniewski and Meyer \cite[Theorem~4.11]{KwasMeyer_EssCrossProd} shows that there is a unique quotient of the larger algebra into which the subalgebra embeds, and the conditional expectation descends to a faithful conditional expectation on this quotient.
	We show that the induced inclusion into the quotient is again maximal abelian, and hence is an essential Cartan pair.
	
	Lastly, we provide constructions of automorphisms of twisted groupoids from automorphisms of the induced essential Cartan pairs and vice versa.
	If the twist is over an effective groupoid, these constructions are mutually inverse and the resulting respective automorphism groups are isomorphic.
	
	\subsection*{Acknowledgements}
	The author would like to thank his doctoral advisor Ralf Meyer for all the assistance and expertise he provided.
	The article consists primarily of results from the author's PhD thesis \cite{Tay_Phd}.
	The author would also like to thank Sven Raum for his valuable insight and feedback to better motivate and contexualise this article, helping demonstrate its relevance in the current landscape of the topic.
	
	\section{$C^*$-algebras of twists over groupoids}
	
	Let $G$ be an \'etale groupoid.
	A \emph{twist} over $G$ is a central extension $\Sigma$ of $G$ by $G^{(0)}\times\TT$; the unit space times the circle group
	$$G^{(0)}\times\TT\hookrightarrow\Sigma\twoheadrightarrow G.$$
	By definition, $\Sigma$ carries a central action of $\TT$.
	From this one may construct a canonical line bundle associated to the twist $(G,\Sigma)$ as $L:=\frac{\Sigma\times\CC}{\TT}$, where the quotient is by the action $z(\sigma,\lambda)=(z\sigma,\bar{z}\lambda)$ for $z\in\TT$, $\sigma\in\Sigma$, and $\lambda\in\CC$.
	Sections of this bundle can be identified with functions $\Sigma\to\CC$ satisfying $f(z\sigma)=\bar{z}f(\sigma)$, and we shall often use these two characterisations interchangeably.
	
	We say an inclusion $A\subseteq B$ of $C^*$-algebras is \emph{non-degenerate} if $A$ contains an approximate unit for $B$.
	Renault defines the Weyl groupoid and Weyl twist in \cite{Renault_Cartan} for non-degenerate inclusions of commutative $C^*$-subalgebras, following the construction of Kumjian in \cite[Theorem~3.1]{Kumjian_CstarDiags}.
	For our purposes we need not alter the construction of the Weyl groupoid and twist to accommodate for our more general definition of essential Cartan pairs.
	
	\begin{definition}[{\cite[1.1]{Kumjian_CstarDiags}}]
		Let $A\subseteq B$ be an inclusion of $C^*$-algebras.
		A \emph{normaliser} of $A$ in $B$ is an element $n\in B$ satisfying $n^*An,nAn^*\subseteq A$.
		We call the inclusion $A\subseteq B$ \emph{regular} if the collection $N(A,B):=\{n\in B: n$ is a normaliser of $A\subseteq B\}$ of normalisers of $A$ inside of $B$ spans a dense subspace of $B$.
	\end{definition}
	
	If \(A\subseteq B\) is non-degenerate then $n^*n$ belongs to $A$ for any normaliser $n\in N(A,B)$ (see for example \cite[Lemma~4.5]{Renault_Cartan}).
	If $A=C_0(X)$ is a commutative non-degenerate $C^*$-subalgebra of $B$ then (following Kumjian \cite{Kumjian_CstarDiags} and Renault \cite{Renault_Cartan}) for any normaliser $n\in N(A,B)$ we can define $\dom(n):=\{x\in X: n^*n(x)>0\}$ and $\ran(n):=\{x\in X:nn^*(x)>0\}$.
	
	Kumjian \cite{Kumjian_CstarDiags} describes how normalisers of an inclusion $A\subseteq B$ of a commutative $C^*$-algebra implement partial homeomorphisms on the Gelfand dual $X$ of $A=C_0(X)$.	
	
	\begin{proposition}[{\cite[1.6]{Kumjian_CstarDiags}}]\label{prop-normaliserAction}
		Let $A=C_0(X)$ be a non-degenerate commutative $C^*$-subalgebra of $B$ and let $n\in N(A,B)$ be a normaliser.
		There exists a unique homeomorphism $\alpha_n:\dom(n)\to\ran(n)$ satisfying
		$$n^*an(x)=a(\alpha_n(x))n^*n(x),$$
		for all $x\in\dom(n)$ and $a\in A$.
	\end{proposition}
	
	\begin{lemma}[{\cite[1.7]{Kumjian_CstarDiags}}]\label{lem-WeylPseudogroupIsAPseudogroup}
		Let $A=C_0(X)$ be a non-degenerate commutative $C^*$-subalgebra of $B$.
		Then
		\begin{enumerate}
			\item if $a\in A$ then $\alpha_a=\Id_{\dom(a)}$;
			\item for $m,n\in N(A,B)$ we have $\alpha_{mn}=\alpha_m\circ\alpha_n$ (wherever the composition is defined) and $\alpha_n^{-1}=\alpha_{n^*}$.
		\end{enumerate}
	\end{lemma}
	
	Let $W\subseteq N(A,B)$ be a closed linear subspace.
	We say $W$ is a \emph{slice} for the inclusion $A\subseteq B$ if $AW+WA\subseteq W$ and $W^*W+WW^*\subseteq A$.
	We observe that $W$ inherits a Hilbert $A$-bimodule structure from the ambient $C^*$-algebra $B$, and we denote by \(s(W)\) (resp. \(r(W)\)) the ideal of \(A\) generated by \(W^*W\) (resp. \(WW^*\)).
	Any two normalisers belonging to the same slice must implement the same partial homeomorphism of $X$ on the intersection of their domains.

	\begin{corollary}\label{cor-partialHomeoSameForASlice}
		Let $A$ be a non-degenerate commutative $C^*$-subalgebra of $B$ and let $W\subseteq B$ be a slice for the inclusion.
		For $m,n\in W$, let $U:=\dom(m)\cap\dom(n)$.
		Then $\alpha_m|_U=\alpha_n|_U$.
		\begin{proof}
			Since $W$ is a slice, the element $m^*n$ belongs to $A$, and so the associated partial homeomorphism $\alpha_{m^*n}$ is the identity map on $\dom(m^*n)$ by Lemma~\ref{lem-WeylPseudogroupIsAPseudogroup}.
			The same lemma implies that $\Id_{\dom(m^*n)}=\alpha_{m^*n}=\alpha_m^{-1}\circ\alpha_n$, whereby $\Id_{\dom(m^*n)}$ and $\alpha_m^{-1}\circ\alpha_n$ have the same domain, namely $U$.
			Thus, on $U$ we have
			\begin{align*}
				\alpha_m|_U&=\alpha_m|_U\circ\Id_{\dom(m^*n)}=\alpha_m|_U\circ\alpha_m^{-1}\circ\alpha_n|_U=\alpha_n|_U.\qedhere
			\end{align*}
		\end{proof}
	\end{corollary}	

	The collection of partial homeomorphisms $\Gg(A):=\{\alpha_n:n\in N(A,B)\}$ forms a \emph{pseudogroup} acting on $X=\hat{A}$, that is, a subsemigroup of the semigroup of partial homeomorphisms of \(X\) which is closed under taking partial inverses and restricting to open subsets.
	
	\begin{definition}[{\cite[Definition~4.2]{Renault_Cartan}}]
		We call $\Gg(A)$ the \emph{Weyl pseudogroup} of the pair $(A,B)$. 
		Define the \emph{Weyl groupoid} $G(A,B)$ of $(A,B)$ as the groupoid of germs of $\Gg(A)$ (see \cite[Section~3]{Renault_Cartan}).
	\end{definition}

	Concretely, the Weyl groupoid is defined as the quotient of $D=\{(\alpha_n,x):n\in N(A,B),x\in\dom(n)\}$ by the equivalence relation
	\begin{align*}
		(\alpha_n,x)\sim(\alpha_m,y)\iff &x=y\text{ and there exists open }U\subseteq X\\
		&\text{ with }x\in U\text{ such that }\alpha_n|_U=\alpha_m|_U.
	\end{align*}
	The topology on $G(A,B)$ is given by basic open sets of the form 
	\begin{equation}\label{eq-UnBisections}
		U_n:=\{[\alpha_n,x]:x\in\dom(n)\}.
	\end{equation}
	The composition of two equivalence classes $[\alpha_n,x]$ and $[\alpha_m,y]$ in $G(A,B)$ is defined if $x=\alpha_m(y)$, and the product is given by $[\alpha_n,x]\cdot[\alpha_m,y]=[\alpha_n\circ\alpha_m,y]$, which by Lemma~\ref{lem-WeylPseudogroupIsAPseudogroup} is equal to $[\alpha_{mn},y]$.
	The inverse is given by $[\alpha_n,x]^{-1}=[\alpha_{n^*},\alpha_n(x)]$.
	The unit space of the Weyl groupoid can and often will be idenitied with the Gelfand dual space \(X=\hat A\) via the map \(X\to G(A,B)^{(0)}\) sending a point \(x\in X\) to the trivial germ \([f,x]\), where \(f\in C_0(X)\) is a function with \(f(x)\neq 0\).
	The germ relation ensures that the class \([f,x]\) does not depend on the choice of function \(f\) since the partial homeomorphism \(\alpha_f\) described in Proposition~\ref{prop-normaliserAction} is obviously trivial when \(f\) commutes with \(A=C_0(X)\).
	
	Renault defines the Weyl twist $\Sigma(A,B)$ by considering pairs $(n,x)$ where $n\in N(A,B)$ and $x\in\dom(n)$ (see \cite[Section~4]{Renault_Cartan}).
	Two pairs $(n,x)$ and $(m,y)$ are equivalent in the twist if $x=y$ and there are functions $a,b\in C_0(X)=A$ with $a(x),b(x)>0$ and $an=bm$.
	The product of equivalence classes $[n,x]$ and $[m,y]$ in the twist is defined whenever the product $[\alpha_n,x]$ and $[\alpha_m,y]$ is defined in the Weyl groupoid, and is given by $[m,x]\cdot[n,y]=[mn,y]$.
	Similarly to the Weyl groupoid, the inverse of $[n,x]$ in the Weyl twist is $[n^*,\alpha_n(x)]$.
	The canonical surjection $\Sigma(A,B)\to G(A,B)$ is given by $[n,x]\mapsto[\alpha_n,x]$, which Renault shows to indeed be a twist over $G(A,B)$ if $A\subseteq B$ is a maximal abelian and non-degenerate inclusion (see \cite[Proposition~4.12]{Renault_Cartan}).\\

	We say an \'etale groupoid $G$ is \emph{effective} if the interior of the isotropy of $G$ is the unit space.
	The following elementary observation will be useful throughout the article.
	\begin{lemma}\label{lem-HDUnitsCloseInIsot}
		Let \(G\) be an \'etale groupoid with Hausdorff unit space \(G^{(0)}\).
		The closure \(\overline {G^{(0)}}\) of the unit space is contained in the isotropy of \(G\).
		\begin{proof}
			Fix an arrow \(\gamma\) in the closure of \(G^{(0)}\) and let \((x_\lambda)\subseteq G^{(0)}\) be a net of units converging to \(\gamma\).
			As the range and source maps are continuous, we see that \(s(x_\lambda)=r(x_\lambda)=x_\lambda\) converges to bot the range and source of \(\gamma\).
			As \(G^{(0)}\) is Hausdorff and contains these limits, they must coincide, so \(\gamma\) belongs to the isotropy of \(G\).
		\end{proof}
	\end{lemma}

	Lemma~\ref{lem-HDUnitsCloseInIsot} is uninteresting when applied to Hausdorff groupoids; the closure of units of a Hausdorff groupoid consists only of units.
	In the non-Hausdorff setting, the effectivity condition for groupoids places a restriction on how large the closure of the unit space can be by insisting that taking the closure of units never adds new points to the interior.
	Since the non-Hausdorffness of a groupoid is detected by whether the unit space is closed (at least in cases where the unit space is Hausdorff), effectivity does place some restriction on how `badly' the groupoids we consider can fail to be Hausdorff.

	The Weyl groupoid is effective by construction; this is shown in \cite[Proposition~3.3]{Renault_Cartan} and \cite[Corollary~3.2.7]{Raad_Phd}.
	We provide a quick proof here for convenience.
	
	\begin{lemma}\label{lem-weylGrpdIsEffective}
		The Weyl groupoid \(G(A,B)\) is effective.
		\begin{proof}
			It suffices to consider basic open sets of the form \(U_n\) for \(n\in N(A,B)\) as given in (\ref{eq-UnBisections}) since such open sets form a base for the topology on \(G(A,B)\), so suppose \(U_n\) is contained in isotropy for some \(n\in N(A,B)\).
			Then for all \(x\in\dom(n)\) we have
			\[
				\alpha_n(x)=r[\alpha_n,x]=s[\alpha_n,x]=x,
			\]
			whereby \(\alpha_n=\Id_{\dom(n)}\).
			The germ relation then gives \([\alpha_n,x]=[\Id_{\dom(n)},x]\in G(A,B)^{(0)}\) for all \(x\in\dom(n)\), whereby \(U_n\subseteq G(A,B)^{(0)}\), and \(G(A,B)\) is effective.
		\end{proof}
	\end{lemma}
	
	Kwa\'sniewski and Meyer define the essential (twisted) groupoid $C^*$-algebra as follows (see \cite[Section~7.2]{KwasMeyer_EssCrossProd}).
	Let $(G,\Sigma)$ be a twist over an \'etale groupoid with locally compact Hausdorff unit space.
	The canonical line bundle \(L=\frac{\Sigma\times\CC}{\TT}\) associated to the twist \(\Sigma\) forms a Fell line bundle over \(G\) with multiplication given by \([z,\sigma]\cdot[w,\eta]=[zw,\sigma\eta]\) and involution given by \([z,\sigma]^*=[\bar{z},\sigma^{-1}]\).

	For each open bisection $U\subseteq G$ let $C_c(U,\Sigma)$ be the vector space of compactly supported continuous sections $U\to L|_U$ of the restriction \(L|_U\) of $L$ to $U$.
	Let $\Cc_c(U,\Sigma)$ be the set of sections in $C_c(U,\Sigma)$, extended by zero to sections $G\to L$.
	That is, an element $f\in\Cc_c(U,\Sigma)$ is a section $G\to L$ with compact support contained in $U$, and the restriction $f|_U$ is continuous on $U$.
	Define 
	\begin{equation}\label{eq-CcG}
		\Cc_c(G,\Sigma)=\sspan_{U\subseteq G}\Cc_c(U,\Sigma),
	\end{equation}
	where $U$ ranges over open bisections of $G$.
	The space $\Cc_c(G,\Sigma)$ carries a convolution product and involution given by
	$$f\ast g(\gamma):=\sum_{r(\gamma)=r(\eta)}f(\eta)\cdot g(\eta^{-1}\gamma),\quad f^*(\gamma)=\overline{f(\gamma^{-1})},$$
	for all $f,g\in\Cc_c(G,\Sigma)$ and $\gamma\in G$. 
	Note that the sum in this convolution product is finite since the functions $f$ and $g$ have compact support, and so have finite support over each $r$-fibre in $G$.
	The \emph{full twisted groupoid $C^*$-algebra} $C^*(G,\Sigma)$ is defined as the maximal $C^*$-completion of the ${}^*$-algebra $\Cc_c(G,\Sigma)$.
	
	The reduced \(C^*\)-algebra of the twist \((G,\Sigma)\) admits multiple equivalent constructions.
	Following Bice \cite{Bice_Sections}, we may define the reduced \(C^*\)-algebra of the twist \((G,\Sigma)\) as the completion of \(\Cc_c(G,\Sigma)\) within the space of all sections of \(L\) under the appropriate reduced norm.
	While the space of all sections of \(L\) is not a \(C^*\)-algebra, Bice shows that the closure of \(\Cc_c(G,\Sigma)\) under an approriate norm (the \(\mathrm{b}\)-norm) is a \(C^*\)-algebra, and moreover, that the formulae for the convolution product and involution given above extend to the closure.
	Bice shows that this construction is consistent with the definitions of reduced twisted groupoid \(C^*\)-algebras considered in other parts of the literature (e.g. \cite{KwasMeyer_EssCrossProd}).

	The traditional definition of the reduced \(C^*\)-algebra of a twist \((G,\Sigma)\) is as a completion of \(\Cc_c(G,\Sigma)\) in a suitable \(C^*\)-norm.
	As a result, elements of the reduced \(C^*\)-algebra are not, a priori, sections of the associated line bundle \(L\), and so the product in \(C^*_\red(G,\Sigma)\) does not immediately admit a convolution formula like the one above.
	This obstruction is typically circumvented by observing that there is a canonical ``\(j\)-map'', which is a contractive linear embedding of the reduced \(C^*\)-algebra into the space of Borel sections of the line bundle, and then proving that the convolution formula makes sense and describes the product in \(C^*_\red(G,\Sigma)\).

	Renault showed \cite[4.2~Proposition]{Renault_GrpdApproach} that the convolution product and involution on \(\Cc_c(G,\Sigma)\) extend to the closure of \(\Cc_c(G,\Sigma)\) under the \(j\)-map when the groupoid is second-countable and Hausdorff, and the twist is implemented by a 2-cocycle \(G\to\TT\).
	Generalisations of this followed, for example, Brown, Fuller, Pitts, and Reznikoff demonstrated that the convolution formula holds for the closure of the image of \(\Cc_c(G,\Sigma)\) under the \(j\)-map if the groupoid \(G\) is Hausdorff \cite[Proposition~2.8]{BrownFullerPittsReznikoff_Graded}, not requiring second-countability or that the twist is implemented by a cocycle.
	Further still, Bardadyn, Kwa\'sniewski, and McKee demonstrate that the convolution formula holds for non-Hausdorff groupoids \cite[Proposition~3.15]{BardKwasMckee_BanachAlgGrpd} for more general Banach-algebraic completions of \(\Cc_c(G,\Sigma)\).

	The advantages of Bice's framework are twofold.
	Firstly, it works in sufficient generality to cover the non-Hausdorff groupoids and twists with which we are concerned in this article.
	Secondly, the \(j\)-map is avoided by constructing the reduced \(C^*\)-algebra as the closure inside an ambient Banach space of sections.

	Elements of the reduced \(C^*\)-algebra of a twist \((G,\Sigma)\) have also been described as functions \(\Sigma\to\CC\) which are \emph{\(\TT\)-contravariant}, that is, \(f(z\sigma)=\bar{z}f(\sigma)\) for all \(z\in\TT\) and \(\sigma\in\Sigma\) (see e.g. \cite[Section~4]{Renault_Cartan} or \cite[Section~2.3]{ArmstrongEtAl_LocalBisHyp}).
	This correspondence sends a section \(f\in C^*_\red(G,\Sigma)\) of the line bundle \(L=\frac{\CC\times\Sigma}{\TT}\) to the function \(\Sigma\to\CC\) given by 
	\[\sigma\mapsto f(p(\sigma)),\]
	where \(p\colon\Sigma\to G\) is the canonical morphism in the data of the twist.
	Indeed, this function is \(\TT\)-contravariant since necessarily \(f(p(\sigma))=[\lambda,\sigma]\) for some \(\lambda\in\CC\), so \(\bar z f(p(\sigma))=\bar z[\lambda,\sigma]=[\bar z\lambda,\sigma]=[\lambda,z\sigma]\). 
	We shall often identify elements of the \(C^*_\red(G,\Sigma)\) with their corresponding functions on \(\Sigma\) in this way.

	\subsection{The essential \(C^*\)-algebra} 

	The reduced \(C^*\)-algebra of a twist over a non-Hausdorff groupoid typically contains discontinuous sections which detect the failure of the groupoid to be Hausdorff.
	A section in \(\Cc_c(G,\Sigma)\) only detects a topologically `small' amount of non-Hausdorffness, since it is a finite sum of functions that are continuous on their open supports, so the only problems that can occur is at the boundaries of these supports.
	Such sets (i.e. boundaries of open sets) are \emph{nowhere dense}, that is, their closure has empty interior.
	Since generic elements of \(C^*_\red(G,\Sigma)\) are limits of such elements, the collection of points where they are discontinuous is contained in a countable union of nowhere dense sets (see \cite[Lemma~7.13]{KwasMeyer_EssCrossProd}).
	Countable unions of nowhere dense sets are called \emph{meagre}.

	The desired quotient of $C^*_\red(G,\Sigma)$ is obtained by identifying sections whose difference has meagre support.
	This ensures that the values of a section in \(C^*_\red(G,\Sigma)\) at its discontinuities is discarded while retaining enough information to preserve the `continuous' parts of the section.
	Exel and Pitts \cite{ExelPitts_AlgNonHDGrpds} first defined the \emph{essential} \(C^*\)-algebra associated to a twist over a second-countable groupoid with sufficiently small isotropy, and Kwa\'sniewski and Meyer \cite{KwasMeyer_EssCrossProd} generalised this to all twists over \'etale groupoids.
	
	\begin{definition}[{\cite[Section~3.5]{Bice_Sections}, cf. \cite[Proposition~7.18]{KwasMeyer_EssCrossProd}}]
		The \emph{essential twisted groupoid $C^*$-algebra} $C^*_\ess(G,\Sigma)$ is defined as the quotient of $C^*_\red(G,\Sigma)$ by the ideal 
		\[J_\mathrm{sing}:=\{f\in C^*_\red(G,\Sigma):f|_{G\setminus M}\equiv 0\text{ for a meagre set }M\subseteq G\}.\]
	\end{definition}

	\begin{remark}\label{rem-singularIdealMeagreSupp}
		Kwa\'sniewski and Meyer define the singular ideal a little more broadly, and show that the singular ideal \(J_{\mathrm{sing}}\) consists of functions which are non-zero on a meagre subset of the groupoid if the groupoid can be covered by countably many bisections (see \cite[Proposition~7.18]{KwasMeyer_EssCrossProd}).
		This condition is made moot by observing that we can always pass to an open subgroupoid which is covered by countably many open bisections and contains the support of a given fixed section.
		If \(H\subseteq G\) is this open subgroupoid there is a canonical inclusion \(C^*_\red(H,\Sigma|_H)\hookrightarrow C^*_\red(G,\Sigma)\) taking a section of the restricted bundle \(\Sigma|_H\) and extending it by zero to an element of \(\Cc_c(G,\Sigma)\) (recall (\ref{eq-CcG})).
		It follows that if a section has support contained in \(H\) then it belongs to the singular ideal for \(G\) if and only if it belongs to the singular ideal for \(H\).
		The support of any fixed element of \(C^*_\red(G,\Sigma)\) is contained in the union of countably many open bisections, and hence the open subgroupoid these bisections generate is covered by countably many bisections.
		By restricting to this open subgroupoid, we see that \(f\) is singular if and only if it is singular in the reduced \(C^*\)-algebra of this open subgroupoid, and \cite[Proposition~7.18]{KwasMeyer_EssCrossProd} shows that this is equivalent to \(f\) being non-zero on a meagre subset of the groupoid.
		This also shows that the inclusion \(C^*_\red(H,\Sigma|_H)\hookrightarrow C^*_\red(G,\Sigma)\) descends to an inclusion \(C^*_\ess(H,\Sigma|_H)\hookrightarrow C^*_\ess(G,\Sigma)\) of essential \(C^*\)-algebras.
	\end{remark}
	
	\begin{remark}
		If \(G\) is Hausdorff then the singular ideal \(J_{\mathrm{sing}}\) is trivial, and so the essential \(C^*\)-algebra \(C^*_\ess(G,\Sigma)\) is just the reduced \(C^*\)-algebra \(C^*_\red(G,\Sigma)\).
		Specifically, for an open bisection \(U\subseteq G\) the (relatively closed) support of a local section \(f\in C_c(U,\Sigma)\) is a compact subset of \(U\), and hence is also compact in \(G\).
		If \(G\) is Hausdorff then this compact subset is also closed in \(G\), so \(f\) extends to a globally continuous section \(G\to L\) by taking zero values outside of \(U\).
		Hence every linear combination of such sections is continuous, and so the definition \ref{eq-CcG} of \(\Cc_c(G,\Sigma)\) shows that every element of \(\Cc_c(G,\Sigma)\) is a continuous compactly supported section.
		A partition of unity argument shows that every compactly supported continuous section of \(L\) can be written as a finite sum of sections supported on bisections of \(G\), so we see that \(\Cc_c(G,\Sigma)\) and \(C_c(G,\Sigma)\) coincide.
		The reduced norm dominates the supremum norm (see e.g. \cite[Proposition~7.10]{KwasMeyer_EssCrossProd}), so the closure of \(C_c(G,\Sigma)\) in the reduced norm consists of continuous sections.
		Finally, we note that a continuous function on \(G\) with meagre support must be globally zero.
	\end{remark}

	\subsection{Aperiodic and Masa inclusions}
	
	The notion of aperiodicity for $C^*$-inclusions was first introduced by Kwa\'sniewski and Meyer in \cite{KwasMeyer_EssCrossProd}, inspired by the work of Kishimoto \cite{Kish_OuterAuts}.
	This condition implies several useful properties such as the generalised intersection property and uniqueness of pseudo-expectation (see \cite{KwasMeyer_EssCrossProd},\cite{KwasMeyer_PseudoExp}).
	The definition of aperiodicity for bimodules was given in \cite[Definition~4.1]{KwasMeyer_AperTopoFree} and we shall recall it here.
	Since we are considering inclusions of commutative $C^*$-algebras and their slices, aperiodicity is equivalent to both topological non-triviality and pure outerness for Hilbert bimodules (see \cite[Theorem~8.1]{KwasMeyer_AperTopoFree}).
	
	Throughout the later parts of this article we shall consider inclusions \(A\subseteq B\) of \(C^*\)-algebras, where \(A=C_0(X)\) is a maximal commutative subalgebra of \(B\).
	We call such a subalgebra a \emph{masa} (acronym for `maximal abelian self-adjoint subalgebra'), and refer to inclusions of masas as \emph{masa inclusions}.
	In our analysis of masa inclusions we shall use results from \cite{KwasMeyer_EssCrossProd}, \cite{KwasMeyer_PseudoExp}, \cite{KwasMeyer_AperTopoFree} regarding aperiodic bimodules and inclusions.
	We give the following definition of aperiodicity from \cite{KwasMeyer_AperTopoFree} as stated there, not assuming that the \(C^*\)-algebra \(A\) is commutative.
	
	\begin{definition}[{\cite{KwasMeyer_AperTopoFree}, \cite[Definitions~5.9,5.14]{KwasMeyer_EssCrossProd}}]
		Let \(A\) be a \(C^*\)-algebra.
		A Banach $A$-bimodule $W$ is \emph{aperiodic} if for all $\eps>0$, non-zero hereditary subalgebras $D\subseteq A$, and elements $w\in W$ there is $a\in D$ with $a\geq 0$, $||a||=1$ such that
		$$||awa||<\eps.$$
		An inclusion $A\subseteq B$ of $C^*$-algebras is \emph{aperiodic} if the Banach $A$-bimodule $B/A$ is an aperiodic bimodule.
	\end{definition}

	We will show that all regular masa inclusions are aperiodic (see Corollary~\ref{cor-masaIncIsAper} below); the following technical results are in service of this goal.
	This will imply that the conditional expectations we consider later in the article will be unique, and there is also a unique largest ideal of the ambient algebra which has trivial intersection with the masa (see \cite[Theorem~3.6]{KwasMeyer_PseudoExp}).

	\begin{lemma}\label{lem-commutesWithEssential}
		Let \(A\subseteq B\) be a non-degenerate inclusion of \(C^*\)-algebras and let \(I\triangleleft A\) be an essential ideal.
		An element \(b\in B\) commutes with \(A\) if and only if it commutes with \(I\).
		\begin{proof}
			If \(b\in B\) commutes with \(A\) then it certainly commutes with \(I\), so we prove the other implication.
			Suppose \(ba=ab\) for all \(a\in I\).
			For all \(c\in A\) and \(a\in I\) we have
			\((bc-cb)a=b(ca)-c(ba)=b(ca)-(ca)b=0\) since \(ca\in I\) and \(b\) commutes with \(I\).
			Hence \(bc-cb\) annihilates \(I\), whereby \(bc-cb=0\) since \(I\) is an essential ideal in \(A\) and \(A\subseteq B\) is a non-degenerate inclusion.
		\end{proof}
	\end{lemma}
	
	\begin{lemma}\label{lem-sliceCommutesIfPartialHomTriv}
		Let \(A\subseteq B\) be a non-degenerate inclusion with \(A=C_0(X)\) commutative and let \(n\in N(A,B)\) be a normaliser.
		Let $M_n:=\clspan AnA$ be the slice of \(A\subseteq B\) generated by $n$, and let \(V\subseteq\dom(n)\) be the interior of the set of points in \(\dom(n)\) that are fixed by \(\alpha_n\).
		Then \(M_n\cdot C_0(V)\) is contained in the commutant of \(A\).
		\begin{proof}
			The set \(M_n\) forms a slice since \(n\) is a normaliser of \(A\).
			Fix \(m\in M_n\cdot C_0(V)\).
			Let \(W:=(X\setminus \dom(m))^\circ\) be the interior of the complement of \(\dom(m)\) in \(X\).
			Then \(\dom(m)\cup W\) is a dense open subset of \(X\) so \(C_0(\dom(m)\cup W)\) is an essential ideal of \(C_0(X)\).
			By Lemma~\ref{lem-commutesWithEssential} it suffices to show that \(m\) commutes with the ideal \(C_0(\dom(m)\cup W)\).
			Since \(\dom(m)\cap W=\emptyset\) we have \(C_0(\dom(m)\cup W)=C_0(\dom(m))\oplus C_0(W)\), so we shall show that \(m\) commutes with the summand ideals.
			For \(f\in C_0(\dom(m))\) we have
			\begin{align*}
				||mf-fm||^2&=||(mf-fm)^*(mf-fm)||\\
				&=\sup_{x\in X}|f^*m^*mf(x)-m^*f^*mf(x)-f^*m^*fm(x)+m^*f^*fm(x)|,
			\end{align*}
			noting that each of the terms in the expansion on the right belongs to \(C_0(X)\) since \(m\) is a normaliser.
			If \(x\in\dom(m)\subseteq V\) we have \(m^*fm(x)=m^*mf(x)\) as \(\alpha_m(x)=x\) by Corollary~\ref{cor-partialHomeoSameForASlice}, and so
			\begin{align*}
				(mf-fm)^*(mf-fm)(x)&=f^*m^*mf(x)-m^*f^*mf(x)-f^*m^*fm(x)+m^*f^*fm(x)\\
				&=2f^*f(x)m^*m(x)-2f^*f(x)m^*m(x)\\
				&=0.
			\end{align*}
			Otherwise, if \(x\notin\dom(m)\) we have
			\begin{align*}
				(mf-fm)^*(mf-fm)(x)&=f^*m^*mf(x)-m^*f^*mf(x)-f^*m^*fm(x)+m^*f^*fm(x)\\
				&=m^*f^*fm(x)
			\end{align*}
			since \(f(x)=0\).
			By the Cohen-Hewitt factorisation theorem, we may write \(m=m'g\) for some \(m'\in M_n\) and \(g\in C_0(\dom(m))\), noting that \(\dom(m)\subseteq V\).
			We then have \(m^*f^*fm(x)=g^*(x)(m'^*f^*fm')(x)g(x)=0\) since \(x\notin \dom(m)\).
			Hence \(m\) commutes with \(f\), so \(m\) lies in the commutant of \(C_0(\dom(m))\).
			
			If \(f\in C_0(W)\) then, using the above Cohen-Hewitt factorisation again, we have \(mf=m'gf=0\) since \(\dom(m)\cap W=\emptyset\).
			It remains to show that \(fm=0\).
			Since \(||fm||^2=||m^*f^*fm||\) we aim to show that \(m^*f^*fm(x)=0\) for all \(x\in\dom(m)\cup W\).
			For \(x\in\dom(m)\) we have
			\[m^*f^*fm(x)=|f(\alpha_m(x))|^2m^*m(x)=|f(x)|^2m^*m(x),\]
			since \(\alpha_m=\alpha_n|_{\dom(m)}=\Id_{\dom(m)}\) by Corollary~\ref{cor-partialHomeoSameForASlice} and \(\dom(m)\) is contained in the set of points fixed by \(\alpha_n\).
			Since the support of \(f\) is contained in \(W\) which is disjoint from \(\dom(m)\), we see in this case that \(m^*f^*fm(x)=0\).
			Now suppose that \(x\notin\dom(m)\).
			Using the Cohen-Hewitt factorisation from above, we have
			\[m^*f^*fm(x)=g^*m'^*f^*fm'g(x)=|g(x)|^2m'^*f^*fm'(x)=0,\]
			since \(g\in C_0(\dom(m))\).
			Hence \(||fm||^2=||f^*m^*mf||=0\) whereby \(mf=0=fm\) as required.
		\end{proof}
	\end{lemma}
	
	The proof of the following Lemma~\ref{lem-masaIncHasAperAction} will show that a slice $M$ with trivial intersection with the subalgebra is \emph{topologically non-trivial}, that is, the induced partial homeomorphism $\widehat{s(M)}\to\widehat{r(M)}$ of the spectrum of $A$ does not restrict to the identity map on any non-empty open subset of $\widehat{s(M)}$.
	Recall here that \(s(M)\) and \(r(M)\) denote the ideals generated by \(M^*M\) and \(MM^*\) respectively.
	For \(M=M_n\) where \(n\) is a normaliser, these ideals are exactly those generated by \(n^*n\) and \(nn^*\) in \(A\).
	In general this may be stronger than the bimodule $M$ being aperiodic (see \cite[Theorem~4.7]{KwasMeyer_PseudoExp}), but for bimodules over Type~I $C^*$-algebras (in particular, commutative $C^*$-algebras) these conditions are equivalent by \cite[Theorem~8.1]{KwasMeyer_AperTopoFree}.

	We also add the following observation before proceeding with further technical lemmata: a slice for the inclusion \(A\subseteq B\) which is contained in \(A\) is exactly an ideal of \(A\).
	Moreover, since the intersection of two slices is again a slice, and since \(A\) is always itself a slice, the intersection of a generic slice with \(A\) picks out the `trivial' part.
	Thus, for a slice \(M\), the subspace \(M\cap A\) is an ideal of \(A\) representing the trivial part, and so the non-trivial part of \(M\) should be (at least mostly) represented by the orthogonal complement of \(M\cap A\) in \(M\), viewing it as a Hilbert bimodule.
	This can be realised by multiplying \(M\) by the annihilator of the source ideal of \(M\cap A\), but since \(M\cap A\) is its own source ideal, this reduces to \(M\cdot (M\cap A)^\perp\), where \((M\cap A)^\perp\) denotes the annihilator ideal of \(M\cap A\) in \(A\).

	\begin{lemma}\label{lem-masaIncHasAperAction}
		Let $A=C_0(X)\subseteq B$ be a regular masa inclusion and let $n\in N(A,B)$ be a normaliser.
		Let $M_n:=\clspan AnA$ be the slice of \(A\subseteq B\) generated by $n$.
		Then the bimodule $M_n\cdot(M_n\cap A)^\perp$ is aperiodic.
		\begin{proof}
			Note first that the inclusion is non-degenerate by \cite[Theorem~2.6]{Pitts_NormalisersApproxIds}.
			Let $V\subseteq\dom(n)$ be the interior of the set of points $x\in\dom(n)$ with $\alpha_n(x)=x$.
			By Lemma~\ref{lem-sliceCommutesIfPartialHomTriv} we see that $M_n\cdot C_0(V)$ commutes with $A$, whereby \(M_n\cdot C_0(V)\subseteq A\) since \(A\) is maximal commutative.
			A slice for the inclusion \(A\subseteq B\) which is contained in \(A\) is exactly an ideal of \(A\), so we have
			\[M_n\cdot C_0(V)=(M_n\cdot C_0(V))^*(M_n\cdot C_0(V))=C_0(V)\cap s(M_n).\]
			Since \(C_0(\dom(n))\) contains \(n^*n\), we observe \(C_0(V)\subseteq C_0(\dom(n))=s(M_n)\), so we have \(M_n\cdot C_0(V)=M_n\cap A\).
			Conversely for any $a\in M_n\cap A$ we have $\alpha_a=\Id_{\dom(a)}$ and \(\alpha_n\) extends \(\alpha_a\) by Corollary~\ref{cor-partialHomeoSameForASlice}, whereby $\dom(a)\subseteq V$ and $a\in C_0(V)$.
			Thus $C_0(V)=M_n\cap A$ and so $(M_n\cap A)^\perp = C_0(V)^\perp=C_0(X\setminus\overline{V})$.
			The partial homeomorphism \(\alpha_n\) restricts to a topologically non-trivial map on \(X\setminus\overline {V}\) which shows that \(M_n\cdot C_0(X\setminus \overline{V})\) is a topologically non-trivial bimodule.
			It follows from \cite[Theorem~4.7]{KwasMeyer_PseudoExp} that \(M_n\cdot (M\cap A)^\perp=M_n\cdot C_0(X\setminus\overline{V})\) is aperiodic.
		\end{proof}
	\end{lemma}
	
	\begin{lemma}\label{lem-aperSpanAperInc}
		Let $A\subseteq B$ be a non-degenerate inclusion of $C^*$-algebras.
		If $B$ is densely spanned by slices $W\subseteq B$ with the property that $W\cdot (W\cap A)^\perp$ is an aperiodic $A$-bimodule then the inclusion $A\subseteq B$ is aperiodic.
		\begin{proof}
			The slices $W$ densely span $B$, so slices $W/(W\cap A)$ have densely spanning image in the $A$-bimodule $B/A$.
			The subbimodule $W/(W\cap A)\cdot ((W\cap A)\oplus(W\cap A)^\perp)\cong W\cdot (W\cap A)^\perp$ is then aperiodic by assumption, and so \cite[Lemma~5.12]{KwasMeyer_EssCrossProd} implies that $W/(W\cap A)$ is aperiodic.
			The set of points satisfying Kishimoto's condition is a closed subspace of $B/A$ by \cite[Lemma~4.2]{KwasMeyer_AperTopoFree}, hence the inclusion $A\subseteq B$ is aperiodic.
		\end{proof}
	\end{lemma}
	
	\begin{corollary}\label{cor-masaIncIsAper}
		Let $A\subseteq B$ be a regular masa inclusion.
		Then $A\subseteq B$ is an aperiodic inclusion.
		\begin{proof}
			Combine Lemmata~\ref{lem-masaIncHasAperAction} and \ref{lem-aperSpanAperInc}.
		\end{proof}
	\end{corollary}

	The converse of Corollary~\ref{cor-masaIncIsAper} fails, we provide an elementary counterexample.
	\begin{example}
		Consider the inclusion \(C(\TT)\subseteq C[0,1]\) by vieweing elements of \(C(\TT)\) as continuous functions \(f\colon[0,1]\to\CC\) satisfying \(f(0)=f(1)\).
		This inclusion is unital and regular, and the quotient Banach \(C(\TT)\)-module \(C[0,1]/C(\TT)\) is isomorphic to \(W:=\CC\) via the map \(f+C(\TT)\mapsto f(1)-f(0)\).
		Note here that the \(C(\TT)\)-module structure on \(W\) is given by \(g\cdot w=w\cdot g=g(0)w\) for \(g\in C(\TT)\) and \(w\in W\).
		A non-zero hereditary subalgebra \(D\subseteq C(\TT)\) is a non-zero ideal as \(C(\TT)\) is commutative, so is of the form \(C_0(U)\) for some non-empty open \(U\subseteq \TT\).
		Pick \(g\in C_0(U)\) with \(||g||=1\) and \(g(0)=0\); this can always be done.
		Then \(g\cdot w\cdot g=g(0)^2w=0\) for all \(w\in W\), so \(W\cong C[0,1]/C(\TT)\) is aperiodic, hence the inclusion \(C(\TT)\subseteq C[0,1]\) is aperiodic.
		However, clearly \(C(\TT)\) is not a maximal commutative subalgebra.
	\end{example}

	The various \(C^*\)-algebras associated to a twist \((G,\Sigma)\) contain a common distinguished subalgebra of sections supported on the unit space \(G^{(0)}\).
	Since the twist is by definition trivial over the unit space, the corresponding line bundle is as well, and sections of the line bundle over the unit space of \(G\) correspond to continuous functions \(G^{(0)}\to\CC\); we gain an isomorphic copy of \(C_c(G^{(0)})\) in the \(C^*\)-algebras associated to the twist as the closure of \(C_c(G^{(0)})\subseteq\Cc_c(G,\Sigma)\) in the various norms.
	When \(G\) is Hausdorff and hence all sections in \(C^*_\red(G,\Sigma)\) are continuous, any section \(f\in C^*_\red(G,\Sigma)\) restricts to a continuous section \(f|_{G^{(0)}}\) of the line bundle \(L|_{G^{(0)}}\) over the unit space \(G^{(0)}\).
	Restriction of sections to the unit space (then considered as elements of \(C_0(G^{(0)})\)) forms a contractive linear map \(E:C^*_\red(G,\Sigma)\to C_0(G^{(0)})\).

	\begin{definition}[{\cite[Definition~3.1]{KwasMeyer_EssCrossProd}}]
		Let \(A\subseteq B\) be an inclusion of \(C^*\)-algebras.
		A \emph{conditional expectation} is a completely positive contractive linear map \(E:B\to A\) such that \(E|_A=\Id_A\).

		Let \(A\subseteq\tilde{A}\) be another inclusion of \(C^*\)-algebras. 
		A \emph{generalised conditional expectation} is a completely positive contractive linear map \(E:B\to\tilde{A}\) such that \(E|_A=\Id_A\).

		We say that \(E\) is \emph{faithful} if \(E(b^*b)=0\) implies \(b=0\) for all \(b\in B\).
	\end{definition}

	For an \'etale Hausdorff groupoid \(G\), there is a faithful conditional expectation \(C^*_\red(G,\Sigma)\to C_0(G^{(0)})\) given by restricting functions.
	This conditional expectation leads to an equivalent characterisation of the reduced \(C^*\)-algebra of a twist as a quotient of the full \(C^*\)-algebra in the following sense:

	\begin{proposition}[{\cite[Definition~2.10, Proposition~2.17]{KwasMeyer_NCCart}, \cite[Proposition~ 7.9]{KwasMeyer_EssCrossProd}}]
		Let \((G,\Sigma)\) be a twist over a locally compact Hausdorff \'etale groupoid.
		Restriction of sections \(C_c(G,\Sigma)\to C_0(G^{(0)})\) extends to a conditional expectation \(E:C^*(G,\Sigma)\to C_0(G^{(0)})\).
		Let \(\Nn_E\subseteq\ker(E)\) be the largest ideal of \(C^*(G,\Sigma)\) contained in the kernel of \(E\).
		Then the quotient \(C^*(G,\Sigma)/\Nn_E\) is isomorphic to the reduced \(C^*\)-algebra \(C^*_\red(G,\Sigma)\).
	\end{proposition}

	Of course, if \(G\) is not Hausdorff, the sections in \(C^*_\red(G,\Sigma)\) are not all continuous, and so the restriction map does not take values in \(C_0(G^{(0)})\).
	The essential \(C^*\)-algebra does admit a similar characterisation as the quotient of the full \(C^*\)-algebra by the largest ideal in the kernel of a generalised conditonal expectation; namely by considering restrictions of sections up to meagre support.
	For a locally compact Hausdorff space \(X\), let \(\bB(X)\) denote the algebra of bounded Borel-measurable functions \(X\to\CC\), and let \(\mM(X)\) denote the ideal of such functions that vanish on a comeagre subset of \(X\) (i.e. functions with meagre support).
	The composition \(\Cc_c(G,\Sigma)\to\bB(G^{(0)})\to\bB(G^{(0)})/\mM(G^{(0)})\) restricting a section to the unit space of \(G\) and then taking its equivalence class in the quotient \(\bB(G^{(0)})/\mM(G^{(0)})\) extends to a completely positive contractive linear map
	\begin{equation}\label{eq-ELDefn}
		EL\colon C^*(G,\Sigma)\to \bB(G^{(0)})/\mM(G^{(0)}),
	\end{equation}
	and the quotient by the largest ideal \(\Nn_{EL}\subseteq\ker(EL)\) is the essential \(C^*\)-algebra \(C^*_\ess(G,\Sigma)\) (see \cite[Definition~7.12, Proposition~7.18]{KwasMeyer_EssCrossProd}).
	This implies that the singular ideal \(J_{\mathrm{sing}}\) coincides with \(\Nn_{EL}\).
	
	By results of Dixmier \cite{Dixmier_espacesStone} and later Gonshor \cite[Theorem~1]{Gonshor_InjHulls2}, Hamana's injective hull (see \cite{Hamana_InjEnvs}) of $C_0(X)$ is isomorphic to \(\bB(X)/\mM(X)\).
	Moreover, the local multiplier algebra coincides with the injective hull for commutative \(C^*\)-algebras \cite[Theorem~1]{Frank_InjEnvsLocMultAlgs}, so we have
	\begin{equation}\label{eq-mlocBorMeag}
		\Mloc(C_0(X))\cong I(C_0(X))\cong \bB(X)/\mM(X).
	\end{equation}
	Hence the generalised conditional expectation \(EL\) which characterises the essential \(C^*\)-algebra of a twist is one taking values in the local multiplier algebra or injective hull of \(C_0(G^{(0)})\).

	\begin{definition}
		Let \(A\subseteq B\) be an inclusion of \(C^*\)-algebras.
		A \emph{local conditional expectation} or \emph{local expectation} is a generalised conditional expectation \(E:B\to\Mloc(A)\) taking values in the local multiplier algebra of \(A\).
	\end{definition}

	By the above discussion, local conditional expectations are the same as generalised conditional expectations taking values in the injective hull of \(A\).
	Such generalised expectations are also called \emph{pseudoexpectations} in the literature (e.g. \cite[Definition~3.1]{KwasMeyer_EssCrossProd} and references therein).
	The injective hull and local multiplier algebra of a \(C^*\)-algebra do not coincide in general (one example is the Calkin algebra, see \cite[Section~5]{Hamana_InjEnvs}) so these definitions are generally different. 
	Throughout this article we shall consider conditional expectations only onto commutative subalgebras, and so may freely use both terms interchangeably.

	\section{Essential Cartan subalgebras of $C^*$-algebras}

	We now define an essential Cartan pair in analogue to Renault's definition \cite[Definition~5.1]{Renault_Cartan}.
	
	\begin{definition}\label{defn-essCommCartanPair}
		Let $A\subseteq B$ be an inclusion of $C^*$-algebras with $A$ commutative.
		We say the pair $(A,B)$ is an \emph{essential Cartan pair} if the following conditions hold:
		\begin{enumerate}[leftmargin=2.5cm]
			\item[(EC1)] $A$ is a regular subalgebra of $B$;
			\item[(EC2)] $A$ is a maximal abelian subalgebra of $B$ (masa);
			\item[(EC3)] there exists a faithful local expectation $E:B\to\Mloc(A)$.
		\end{enumerate}
	\end{definition}
	
	We may also use the phrase \emph{essential Cartan inclusion} to refer to an essential Cartan pair.
	
	There are some notable differences in this definition to the definition of Renault in \cite{Renault_Cartan}.
	Firstly, the condition that $A$ is a non-degenerate subalgebra of $B$ has been shown to be redundant by Pitts \cite[Theorem~2.7]{Pitts_NormalisersApproxIds}; all regular masa inclusions are automatically non-degenerate.
	Secondly, the conditional expectation need not take values in $A$, but may take values in $\Mloc(A)$.
	Since $A\subseteq B$ is assumed to be regular and masa, Corollary~\ref{cor-masaIncIsAper} implies that $A\subseteq B$ is an aperiodic inclusion, and so $E$ is the unique pseudo-expectation for the inclusion by \cite[Proposition~8.2]{KwasMeyer_PseudoExp}.
	Lastly, we do not require that the larger \(C^*\)-algebra \(B\) is separable.
	The reconstruction theorem of Renault \cite[Theorem~5.6]{Renault_Cartan} was extended to the non-separable case by Kwa\'sniewski and Meyer in \cite{KwasMeyer_NCCart}, and also by Raad in \cite{Raad_GeneraliseRenault}.

	Throughout the rest of this section, we shall fix a \textbf{regular masa inclusion $A\subseteq B$}.

	Pitts defines \emph{virtual Cartan inclusions} as regular masa inclusions \(A\subseteq B\) such that \(A\) detects ideals in \(B\) (see \cite[Definition~1.1]{Pitts_StructRegInc1}).
	While not superficially obvious, this concept coincides with Definition~\ref{defn-essCommCartanPair}, which we show now.

	\begin{lemma}\label{lem-virtCartEquiv}
		An inclusion \(A\subseteq B\) is a virtual Cartan inclusion in the sense of Pitts \cite[Definition~1.1]{Pitts_StructRegInc1} if and only if it is an essential Cartan inclusion.
		\begin{proof}
			First suppose that \(A\subseteq B\) is a virtual Cartan inclusion.
			There is then a unique pseudoexpectation \(E\) gained by combining Corollary~\ref{cor-masaIncIsAper} and \cite[Proposition~8.2]{KwasMeyer_EssCrossProd}.
			This pseudo-expectation is symmetric since it must agree with the canonical \(\Mloc(A)\)-expectation on \(A\rtimes S\), the full crossed product arising from the canonical action of the slice inverse semigroup \(S:=\Ss(A,B)\) on \(A\).
			If \(A\) detects ideals in \(B\) then the kernel of \(E\) cannot contain any non-zero ideal, so \(E\) is faithful by \cite[Corollary~3.8]{KwasMeyer_EssCrossProd}, and so \(A\subseteq B\) fulfils condition (iii) of Definition~\ref{defn-essCommCartanPair} and so is an essential Cartan inclusion.

			Conversely, \cite[Proposition~5.8]{KwasMeyer_EssCrossProd} together with \cite[Corollary~3.8]{KwasMeyer_EssCrossProd} shows that essential Cartan inclusions must detect ideals, and hence are virtual Cartan inclusions.
		\end{proof}
	\end{lemma}

	The terminology of Pitts predates the earliest version of this article.
	To avoid potential confusion for returning readers, we keep the terminology of \emph{essential Cartan pair} in this and other articles, the first versions of which used this terminology.
	
	\subsection{The evaluation map}
	Renault defines an evaluation map $B\to C^*_\red(G,\Sigma)$ directly taking elements of $B$ to sections of the canonical line bundle associated to $\Sigma$.
	An integral part to this construction is that the conditional expectation takes values in the subalgebra $A=C_0(X)$, and so elements may be considered as functions on $X$.
	This approach breaks down in our setting: elements of $\Mloc(C_0(X))\cong\bB(X)/\mM(X)$ (recall (\ref{eq-mlocBorMeag})) are not functions on $X$, but rather equivalence classes of functions up to meagre support.
	Since singletons are (often) meagre in $X$, elements of $\Mloc(A)$ cannot be evaluated at points.
	We circumvent this problem by showing that normalisers for the inclusion $A\subseteq B$ can be represented by functions defined on a dense open subset, and so give a well-defined class in the Borel-mod-meagre setting since $X$ is a Baire space. 
	
	\begin{lemma}\label{lem-condExpPreservesSlices}
		Let $A\subseteq B$ be regular masa inclusion.
		Then the inclusion $A\subseteq B$ is aperiodic, and hence there is a unique local conditional expectation $E:B\to\Mloc(A)$, and $E(W)\subseteq M(W\cap A)$ for any slice $W\subseteq B$ of the inclusion.
		\begin{proof}
			Write \(A=C_0(X)\) and identify \(\bB(X)/\mM(X)\cong I(C_0(X))\cong\Mloc(A)\) using (\ref{eq-mlocBorMeag}).
			By the universal property of the injective hull, we may always find at least one local conditional expectation \(E\colon B\to \Mloc(A)\).
					
			By Corollary~\ref{cor-masaIncIsAper} the inclusion $A\subseteq B$ is aperiodic, so has at most one pseudoexpectation by \cite[Propositions~3.9,3.16]{KwasMeyer_PseudoExp}.
			Lemma~\ref{lem-masaIncHasAperAction} states that for any slice $W\subseteq B$, the subslice $W\cdot(W\cap A)^\perp$ is an aperiodic $A$-bimodule.
			Then \cite[Lemma~5.10]{KwasMeyer_EssCrossProd} gives that $\Mloc(A)$ contains no non-zero aperiodic bimodules, and the image of an aperiodic bimodule under a bounded bimodule map is aperiodic by \cite[Lemma~5.12]{KwasMeyer_EssCrossProd}. 
			Thus $E(W\cdot(W\cap A)^\perp)=\{0\}$.
			
			For a normaliser $n\in N(A,B)$ and a slice $W\subseteq B$ with $n\in W$, the local multiplier $E(n)$ is determined by how it acts on the essential ideal $(W\cap A)\oplus (W\cap A)^\perp$.
			Fix $a\in W\cap A$ and $a^\perp\in(W\cap A)^\perp$.
			Then $E(n)(a+a^\perp)=E(na)+E(na^\perp)$.
			Since $na\in W\cdot (W\cap A)=W\cap A\subseteq A$, we have $E(na)=na$, and since $na^\perp\in W\cdot (W\cap A)^\perp$, we have $E(na^\perp)=0$.
			Thus $E(n)$ is a multiplier on $W\cap A$, identifying $M(W\cap A)\subseteq\Mloc(W\cap A)\subseteq\Mloc(A)$ via the canonical inclusions.
		\end{proof}
	\end{lemma}
	
	\begin{corollary}\label{cor-expOfNormaliserIsFunction}
		For each normaliser $n\in N(A,B)$ the element $E(n)\in\Mloc(A)$ is represented by a continuous bounded function on a dense open subset of $X=\hat{A}$.
		Moreover, the images of finite sums of normalisers under $E$ are represented by continuous bounded functions on dense open subsets of $X$.
		\begin{proof}
			By Lemma~\ref{lem-condExpPreservesSlices} we have that $E(n)$ belongs to $M(I)$ for some essential ideal $I\triangleleft A$.
			Essential ideals of $A=C_0(X)$ are of the form $I=C_0(U)$ for dense open subsets $U\subseteq X$.
			This, coupled with the identification of $M(C_0(U))$ with the algebra $C_\mathrm{b}(U)$ of bounded continuous functions on $U$ allows us to express $E(n)$ as a bounded continuous function on $U$.
			
			For a finite linear collection of normalisers $n_i\in N(A,B)$, we have $E(\sum_i n_i)=\sum_i E(n_i)$, where each $E(n_i)$ is a function on a dense open subset $U_i\subseteq X$.
			The sum of these functions defines a continuous and bounded function on the intersection $\bigcap_i U_i$, which is again dense and open as the intersection of finitely many dense open subsets.
		\end{proof}
	\end{corollary}

	Given a normaliser \(n\in N(A,B)\), it will be important to describe the value \(EL(n)\) of the conditional expectation (\ref{eq-ELDefn}) at \(n\).
	The following corollary allows us to canonically pick a function on \(X=\hat{A}\) to represent \(EL(n)\) in the quotient \(\Mloc(A)=\bB(X)/\mM(X)\) (under the identification of (\ref{eq-mlocBorMeag})).
	Pitts first proves this in \cite[Corollary~3.6]{Pitts_StructRegInc1} using Frol\'ik decompositions of normalisers in the injective hull of \(A\) (which, in the commutative setting, coincides with the local multiplier algebra).
	We offer an alternative proof that does not require the machinery of Frol\'ik decompositions.
	
	\begin{corollary}[cf. {\cite[Corollary~3.6]{Pitts_StructRegInc1}}]\label{cor-explicitFuncRepExpOfNormaliser}
		Let $n\in N(A,B)$ be a normaliser with \(A=C_0(X)\) commutative, and let 
		\begin{equation}\label{eq-intTrivDom}
			V_n:=\{x\in \dom(n):\alpha_n(x)=x\}^\circ
		\end{equation}
		be the interior of the set of fixed points for $\alpha_n$.
		Identify \(\Mloc(C_0(X))\) with \(\bB(X)/\mM(X)\) as in \textup{(\ref{eq-mlocBorMeag})}.
		Then $E(n)$ is represented in $\Mloc(A)$ by a bounded continuous function on $V_n^\ess:=V_n\cup(X\setminus\overline{V_n})^\circ$ which we denote by $n|_{V^\ess_n}$, and for any $g\in C_0(V_n)$ we have $n|_{V^\ess_n}g=E(ng)=ng$, and for any $g^\perp\in C_0(V_n)^\perp$ we have $n|_{V^\ess_n}g=E(ng)=0$.
		If $b=\sum_i n_i$ is a sum of finitely many normalisers $n_i\in N(A,B)$ then $E(b)$ is represented in $\Mloc(A)$ by a function defined on a dense open subset $V^\ess_b:=\bigcap_i V^\ess_{n_i}$.
		\begin{proof}
			Let $W_n:=\clspan AnA$ be the smallest slice containing $n$.
			Then $s(W_n)$ is exactly $\clspan An^*AnA=C_0(\dom(n))$, in which $C_0(V_n)\oplus C_0((X\setminus\overline{V_n})^\circ\cap\dom(n))$ is an essential ideal.
			Since $\alpha_n|_{V_n}=\Id_{V_n}$ we have that $n$ commutes with $C_0(V_n)$, whereby $n$ restricts to a multiplier of $C_0(V_n)$ which is represented by a continuous and bounded function on $V_n$.
			Conversely for all $x\in(X\setminus\overline{V_n})^\circ$, we have that $\alpha_n(x)\neq x$, whereby $W_n\cdot C_0((X\setminus\overline{V_n})^\circ)$ is a topologically non-trivial Hilbert $A$-bimodule.
			The subbimodule $W_n\cdot C_0((X\setminus\overline{V_n})^\circ)$ is then purely outer by \cite[Theorem~8.1]{KwasMeyer_AperTopoFree}, whereby $W_n\cdot C_0((X\setminus\overline{V_n})^\circ)\cap A=\{0\}$.
			Thus $E(n)$ gives a multiplier of $C_0(V_n)$ by Lemma~\ref{lem-condExpPreservesSlices}, which extends by zero to a multiplier of $C_0(V_n)\oplus C_0((X\setminus\overline{V_n})^\circ)=C_0(V^\ess_n)$, which is in turn given by a bounded function $n|_{V^\ess_n}$ on $V^\ess_n$.
			Then for any $g\in C_0(V_n)$ we have $n|_{V^\ess_n}g=E(ng)=ng$ and for any $g^\perp\in C_0(V_n)^\perp$ we have $n|_{V^\ess_n}g^\perp=E(ng^\perp)=0$.
			
			The result about finite sums of normalisers follows since the intersection of finitely many dense open subsets of a topological space is again dense and open.
		\end{proof}
	\end{corollary}

	The set \(V_n\) defined in (\ref{eq-intTrivDom}) is the largest open set on which the partial homeomorphism \(\alpha_n\) is trivial, and so the germs of \(\alpha_n\) on \(V_n\) are all trivial germs.
	
	\begin{lemma}\label{lem-vnTrivPartUn}
		For a normaliser \(n\in N(A,B)\), let \(U_n\) be as in \textup{(\ref{eq-UnBisections})} and let \(V_n\) be as in \textup{(\ref{eq-intTrivDom})}.
		Then \(V_n = U_n\cap G(A,B)^{(0)}\).
		\begin{proof}
			The set $U_n\cap G^{(0)}$ is contained in $V_n$ since $\alpha_n(x)=x$ for all $x\in U_n\cap G^{(0)}$ and $U_n\cap G^{(0)}$ is an open set.
			Conversely, if $x\in V_n$ then there is a neighbourhood of $x$ contained in $\dom(n)$ which is fixed by $\alpha_n$, and so $[\alpha_n,x]=[\Id,x]\in U_n\cap G^{(0)}$, hence $V_n=U_n\cap G^{(0)}$.
		\end{proof}
	\end{lemma}
	
	In Corollary~\ref{cor-explicitFuncRepExpOfNormaliser} a different choice of normalisers giving rise to $b\in\sspan N(A,B)$ may give a different dense open subset $V_b^\ess\subseteq X$.
	However a differing choice of normalisers summing to $b$ can only change the resulting set $V_b^\ess$ on a meagre set, since $V_b^\ess$ is always dense and open.
	For our purposes we only need that at least one such $V_b^\ess$ exists.

	Our goal is to identify normalisers of the inclusion \(A\subseteq B\) with corresponding elements of \(\Cc_c(U,\Sigma(A,B))\) for open bisections \(U\subseteq G(A,B)\).
	However, general normalisers need to satisfy the `compact support' condition; this fails if \(A=C_0(X)\) is non-unital, so no element of \(C_0(X)\setminus C_c(X)\) can be represented by a section in \(\Cc_c(G(A,B),\Sigma(A,B))\); recall the definition (\ref{eq-CcG}).
	We circumvent this issue using the same strategy as Renault in \cite[Lemma~5.5]{Renault_Cartan}.
	
	\begin{lemma}[{\cite[Lemma~5.5]{Renault_Cartan}}]\label{lem-compactNormalisers}
		Define \(N_c(A,B):=\{n\in N(A,B): n^*n\in C_c(X)\subseteq A\}\).
		Then \(N_c(A,B)\) is dense in \(N(A,B)\), and hence the span of \(N_c(A,B)\) is a dense subspace of \(B\).
		\begin{proof}
			By the Cohen-Hewitt factorisation theorem, any normaliser \(n\in N(A,B)\) can be expressed as the product \(mf\) for some \(m\in N(A,B)\) and \(f\in C_0(X)\).
			Let \(f_k\in C_c(X)\) be a sequence of functions with compact support approximating \(f\) in the supremum norm.
			Then \(n_k:=mf_k\) is a normaliser for each \(k\), the terms \(n_k^*f^*fn_k\) are compactly supported functions on \(X\) for each \(k\), and \(n_k\to n\) in norm.
			Hence \(N_c(A,B)\) is dense in \(N(A,B)\).
		\end{proof}
	\end{lemma}

	Lemma~\ref{lem-compactNormalisers} allows us to restrict our analysis to normalisers satisfying this compact support criterion.
	We now show that these define elements of \(\Cc_c(U,\Sigma(A,B))\) for open bisections \(U\subseteq G(A,B)\).

	In his paper \cite[Lemma~5.1]{Renault_Cartan}, Renault defines functions \(\Sigma(A,B)\to\CC\) directly for all elements of the \(C^*\)-algebra \(B\).
	This works since the conditional expectation in that setting takes values in the subalgebra, and hence the image of an element under the conditional expectation is a function which can be evaluated at a point.
	Since we only have an \(\Mloc(A)\)-valued conditional expectation, we cannot evaluate \(E(b)\) at points, but using Corollary~\ref{cor-explicitFuncRepExpOfNormaliser} we can select representatives with which to work.

	\begin{lemma}\label{lem-localEvaluationMap}
		Let $n\in N_c(A,B)$ and let $U_n$ be the open bisection of $G(A,B)$ determined by $n$ as in \textup{(\ref{eq-UnBisections})}.
		Viewing elements of \(\Cc_c(G(A,B),\Sigma(A,B))\) as \(\TT\)-contravariant functions \(\Sigma(A,B)\to\CC\), there is an element $\hat{n}\in\Cc_c(U_n,\Sigma(A,B))$ defined by
		$$\hat{n}[m,x]=\begin{cases}\frac{(m^*n)|_{V^\ess_{m^*n}}(x)}{\sqrt{m^*m(x)}},& x\in V^\ess_{m^*n},\\
			0,&\text{otherwise,}
		\end{cases}$$
		where \(V_{m^*n}\) is as in \textup{(\ref{eq-intTrivDom})}.
		If $b=\sum_i n_i$ is a sum of finitely many normalisers then there is a function $\hat{b}\in\Cc_0(G(A,B),\Sigma(A,B))$ given by
		$$\hat{b}[m,x]=\begin{cases}\frac{(m^*b)|_{V^\ess_{m^*b}}(x)}{\sqrt{m^*m(x)}},& x\in V^\ess_{m^*b},\\
			0,&\text{otherwise,}
		\end{cases}$$
		which agrees with $\sum_i\widehat{n_i}$ on $U_m\cdot V^\ess_{m^*b}$ for each normaliser $m\in N(A,B)$. Hence \(\hat b\) and \(\sum_i\widehat{n_i}\) agree on a dense open subset of $G(A,B)$.
		\begin{proof}
			We first show that the proposed formula for \(\hat{n}\) is well-defined.
			For \(x\in \dom(n)\) and \(m,m'\in N(A,B)\) such that \([m,x]=[m',x]\).
			If \([\alpha_m,x]=[\alpha_{m'},x]\) does not belong to \(U_n\) then neither \(\alpha_{m^*n}\) nor \(\alpha_{m'^*n}\) are the identity on an open neighbourhood of \(x\), implying \(x\notin V_{m^*n}\cup V_{m'^*n}\). 
			The proposed formula then evaluates to zero for both representatives of the class \([m,x]\).
			Otherwise we have \(x\in V_{m^*n}\cap V_{m'^*n}\), and since \([m,x]=[m',x]\) there are functions \(a,a'\in A=C_0(X)\) with \(a(x),a'(x)>0\) such that \(ma=m'a'\).
			We then compute
			\[\frac{(m'^*n)|_{V_{m'^*n}}(x)}{\sqrt{m'^*m'(x)}}=\frac{a'(x)}{a'(x)}\frac{(m'^*n)|_{V_{m'^*n}}(x)}{\sqrt{m'^*m'(x)}}=\frac{a(x)}{a(x)}\frac{(m^*n)|_{V_{m^*n}}(x)}{\sqrt{m^*m(x)}}=\frac{(m^*n)|_{V_{m^*n}}(x)}{\sqrt{m^*m(x)}},\]
			so the proposed formula is well-defined.
			We note that \(\hat{n}\) takes non-zero values only in \(\Sigma(A,B)|_{U_n}\), since for other values of \(x\) either \(x\notin V_{m^*n}\), for which the formula is zero, or \((m^*n)|_{V^\ess_{m^*n}}(x)=0\).
			Moreover the formula is continuous and \(\TT\)-contravariant on \(\Sigma(A,B)|_{U_n}\) since every element of \(\Sigma(A,B)|_{U_n}\) is of the form \([zn,x]\) for \(z\in\TT\) and \(x\in\dom(n)\), and for such values the formula simplifies to 
			\[\hat{n}[zn,x]=\bar z \sqrt{n^*n}(x),\]
			and the right-hand side has the desired properties.

			That $\hat{b}$ is well-defined follows from the observation that $(m^*n_i)|_{V^\ess_{m^*b}}+(m^*n_j)|_{V^\ess_{m^*b}}$ is exactly $(m^*(n_i+n_j))|_{V^\ess_{m^*b}}$ for all $m\in N(A,B)$ and for any summands $n_i,n_j$ in the expression $b=\sum_i n_i$.
			Since $V^\ess_{m^*b}\subseteq V^\ess_{m^*n_i}$ for each $i$ we see that $(m^*b)|_{V^\ess_{m^*b}}(x)=\sum_i(m^*n_i)|_{V^\ess_{m^*n_i}}(x)$, whereby $\hat{b}$ and $\sum_i\widehat{n_i}$ agree for values of $x\in V^\ess_{m^*b}$.
			The union of $U_m\cdot V^\ess_{m^*b}$ over all normalisers $m\in N(A,B)$ is then a dense open subset of $G(A,B)$ since the sets $U_m$ form a basis for the topology on $G(A,B)$, and each $U_m\cdot V^\ess_{m^*b}$ is open.
		\end{proof}
	\end{lemma}

	For any twist \((G,\Sigma)\) over an \'etale groupoid and for any open bisection \(U\subseteq G\), the space \(\Cc_c(U,\Sigma)\) carries a natural  \(C_0(G^{(0)})\)-bimodule structure given by
	\[f n g(\gamma)=f(r(\gamma))n(\gamma)g(s(\gamma)),\]
	for \(n\in \Cc_c(U,\Sigma)\) and \(f,g\in C_0(G^{(0)})\).
	Identifying \(G(A,B)^{(0)}\) with the Gelfand spectrum \(X\) of \(A\), the following lemma shows that this bimodule structure is compatible with the inherited bimodule structur of a slice of the inclusion \(A\subseteq B\).

	\begin{lemma}\label{lem-compactlySupportedSectionsAsEvalutatedNormalisers}
		Let $n\in N(A,B)$ be a normaliser with \(A=C_0(X)\) commutative, and let \(U_n\) be the open bisection of \(G(A,B)\) in \textup{(\ref{eq-UnBisections})}.
		Any section $f\in\Cc_c(U_n,\Sigma(A,B))$ is of the form $\hat{n} h$ for a function $h\in C_0(X)$.
		Moreover, the assignment $\phi_n:\hat{n} h\mapsto nh$ is a well-defined $C_0(X)$-bimodule map.
		\begin{proof}
			Existence of $h\in C_0(X)$ with $f=\hat{n} h$ follows since $\hat{n}$ is a non-vanishing section over $U_n$, and so since the support of $f$ is contained in $U_n$, there exists $h\in C_0(X)$ with compact support contained in $s(U_n)$ satisfying $\hat{n}h=f$.
			
			To see that the map $\hat{n} h\mapsto nh$ is well-defined, suppose $f=\hat{n}h=\hat{n}g$.
			Then \mbox{$\hat{n}(h-g)=0$}, so for all $[m,x]\in\Sigma(A,B)$ with $[\alpha_m,x]\in U_n$ we have 
			\[0=\hat{n}[m,x](h-g)(x),\]
			which implies that $h(x)=g(x)$ for all $x\in s(U_n)=\dom(n)$ as $\hat{n}$ is a non-vanishing section on $U_n$.
			Then 
			$$E((n(h-g))^*n(h-g))=\overline{(h-g)}E(n^*n)(h-g)=\overline{(h-g)}n^*n(h-g)$$ 
			and for all $x\in X$ we have
			$$\overline{(h-g)}n^*n(h-g)(x)=|h(x)-g(x)|^2n^*n(x)=0,$$
			as either $x\in\dom(n)=\supp(n^*n)$ where $h(x)-g(x)=0$, or $x\notin\dom(n)$ so $n^*n(x)=0$.
			
			We now show that \(\phi\) is a \(C_0(X)\)-bimodule homomorphism.
			For $f_1,f_2\in C_0(X)$, $\hat{n}h\in\Cc_c(U_n,\Sigma(A,B))$, $m\in N(A,B)$, and $x\in\dom(m)\cap s(U_n)$ we have
			$$(f_1\hat{n}hf_2)[m,x]=f_1(\alpha_n(x))\frac{(m^*n)|_{V_{m^*n}}(x)}{\sqrt{m^*m(x)}}h(x)f_2(x),$$
			where \(V_{m^*n}\) is as in (\ref{eq-intTrivDom}).
			Since $h$ has compact support $K\subseteq X$, the function $(f\circ\alpha_n)|_{\dom(n)\cap K}$ extends to some $F\in C_0(X)$ by Tietze's extension theorem, and we have $f_1\hat{n}h=\hat{n}hF$.
			We claim that $f_1nh=nhF$.
			To see this, we note that for $x\in\dom(n)\cap K$ we have
			\begin{align*}
				(f_1nh-n&hF)^*(f_1nh-nhF)(x)\\
				&=(h^*n^*f_1^*f_1nh-F^*h^*n^*f_1nh-h^*n^*f_1^*nhF+F^*h^*n^*nhF)(x)\\
				&=2|f_1(\alpha_n(x))|^2n^*n(x)|h(x)|^2-2|f_1(\alpha_n(x))|^2n^*n(x)|h(x)|^2\\
				&=0,
			\end{align*}
			as \(F\) extends \(f_1\circ\alpha_n\).
			Alternatively if $x\notin\dom(n)\cap K$ then all the terms in the above calculation evaluate to zero as $n^*n(x)$ is a common factor in each term.
			Thus 
			$$\phi_n(f_1\hat{n}hf_2)=\phi_n(\hat{n}hFf_2)=nhFf_2=f_1nhf_2=f_1\phi_n(\hat{n}h)f_2,$$ 
			so $\phi_n$ is a $C_0(X)$-bimodule homomorphism.
		\end{proof}
	\end{lemma}
	
	\begin{lemma}\label{lem-mapsToBAndCcGTwist}
		For $n\in N(A,B)$ let $\phi_n$ be the $C_0(X)$-bimodule homomorphism in Lemma~\textup{\ref{lem-compactlySupportedSectionsAsEvalutatedNormalisers}}.
		Let $D:=\bigoplus_{n\in N(A,B)}^{\alg}\Cc_c(U_n,\Sigma(A,B))$ and define $\Phi:D\to B$ by 
		$$\Phi((f_n)_{n\in N(A,B)})=\sum_{n\in N(A,B)}\phi_n(f_n).$$
		Let $c:D\to\Cc_c(G(A,B),\Sigma(A,B))$ be the map $c((f_n)_{n\in N(A,B)})=\sum_{n\in N(A,B)}f_n$.
		There is a ${}^*$-algebra structure on $D$ such that both $\Phi$ and $c$ are ${}^*$-homomorphisms.
		Moreover, $c$ is surjective and $\Phi$ has dense range in $B$.
		\begin{proof}
			For $f\in\Cc_c(U_n,\Sigma(A,B))$ and $g\in\Cc_c(U_m,\Sigma(A,B))$ in $D$, write $f=\hat{n}h$ and $g=\hat{m}k$ with $h,k\in C_0(X)$ as in Lemma~\ref{lem-compactlySupportedSectionsAsEvalutatedNormalisers}.
			Using Lemma~\ref{lem-localEvaluationMap}, define the product $fg=\widehat{nhm}k\in\Cc_c(U_{nhm},\Sigma(A,B))$, extending bilinearly to all of $D$, and define the involution $(\hat{n}h)^*=h^*\widehat{n^*}\in\Cc_c(U_{n^*},\Sigma(A,B))$, again extending anti-linearly.
			A brief computation shows that this gives $D$ a ${}^*$-algebra structure.\\
			
			Fix $f=\hat{n}h\in\Cc_c(U_n,\Sigma(A,B))$ and $g=\hat{m}k\in\Cc_c(U_m,\Sigma(A,B))$.
			To see that \(\Phi\) is multiplicative, we compute 
			$$\Phi(\hat{n}h\hat{m}k)=\Phi(\widehat{nhm}k)=\phi_{nhm}(\widehat{nhm}k)=nhmk=\phi_n(\hat{n}h)\phi_m(\hat{m}k)=\Phi(f)\Phi(g).$$
			Since $f$ has compact support there is $j\in C_0(X)$ with compact support such that $f=j\hat{n}h$, and so $\Phi(f^*)=\phi_{n^*}(\bar{h}\widehat{n^*}\bar{j})=\bar{h}\phi_{n^*}(\widehat{n^*}\bar{j})=(j\hat{n}h)^*=\Phi(f)^*$.
			Thus $\Phi$ is a ${}^*$-homomorphism.\\
			
			To see that $c$ is a ${}^*$-homomorphism, fix $\hat{n}h\in\Cc_c(U_n,\Sigma(A,B))$ and $\hat{m}k\in\Cc_c(U_m,\Sigma(A,B))$.
			Considering $\hat{n}h$ and $\hat{m}k$ as sections of the canonical line bundle associated to the twist $(G(A,B),\Sigma(A,B))$, for $[\alpha_p,x]\in G(A,B)$ we have
			\begin{align*}
				(c(\hat{n}h)c(\hat{m}k))[\alpha_p,x]&=\sum_{\alpha_q(y)=\alpha_p(x)}\hat{n}h[\alpha_q,y]\hat{m}k[\alpha_{q^*p},x]\\
				&=\hat{n}[\alpha_n,y]h(\alpha_{n^*p}(x))\hat{m}[\alpha_{n^*p},x]k(x)\\
				&=\begin{cases}
					\hat{n}[\alpha_n,\alpha_m(x)]h(\alpha_m(x))\hat{m}[\alpha_m,x]k(x),&[\alpha_{p},x]=[\alpha_{nm},x],\\
					0,&\text{otherwise,}
				\end{cases}\\
				&=c(\widehat{nhm}k)[\alpha_p,x]
			\end{align*}
			The second and third equalities hold since $\hat{n}$ and $\hat{m}$ have support contained in $U_n$ and $U_m$, respectively.
			For $f=\hat{n}h\in\Cc_c(U_n,\Sigma(A,B))$ we have
			\begin{align*}
				c(f^*)[\alpha_p,x]&=\bar{h}\widehat{n^*}[\alpha_p,x]\\
				&=\overline{(h\hat{n})[\alpha_{p^*},\alpha_p(x)]}\\
				&=\begin{cases}
					\overline{h(x)\hat{n}[\alpha_n,\alpha_{n^*}(x)]},&[\alpha_p,x]=[\alpha_{n^*},x],\\
					0,&\text{otherwise,}
				\end{cases}\\
				&=(\hat{n}h)^*[\alpha_{p},x]\\
				&=c(f)^*[\alpha_p,x]
			\end{align*}
			thus $c$ is a ${}^*$-homomorphism.\\
			
			That $c$ is surjective is clear as $\Cc_c(G(A,B),\Sigma(A,B))$ is spanned by the spaces $\Cc_c(U,\Sigma(A,B))$ ranging over open bisections $U\subseteq G$, and bisections of the form $U_n$ (as given in \ref{eq-UnBisections}) for normalisers $n\in N(A,B)$ form a basis for the topology on $G(A,B)$.
			
			To see that $\Phi$ has dense range, it suffices to show that any normaliser $n\in N(A,B)$ can be approximated by elements in the image of $\Phi$ since $A\subseteq B$ is a regular subalgebra.
			For $n\in N(A,B)$ we can write $n=mh$ for $m\in N(A,B)$ and $h\in C_0(X)=A$ by the Cohen-Hewitt factorisation theorem.
			Since $C_c(X)$ is dense in $C_0(X)$, we have compactly supported functions $h_k$ with $h_k\to h$ uniformly.
			Since $n^*n$ has support contained in $\dom(n)=s(U_n)$, we can pick functions $h_k$ with support contained in $s(U_n)$.
			Then $n=mh=\lim\limits_{k\to\infty}mh_k$, and each $mh_k$ is exactly the image of $\hat{m}h_k\in\Cc_c(U_m,\Sigma(A,B))$.
			Thus the image of $\Phi$ is dense in $B$.
		\end{proof}
	\end{lemma}

	Lemma~\ref{lem-mapsToBAndCcGTwist} allows us to model a dense subspace of elements of \(B\) by genuine functions, and gives as a similar tool to the evaluation map of Renault \cite[Lemma~5.1]{Renault_Cartan}.
	This culminates in Corollary~\ref{cor-mapDescendsToCc(Gtwist)} and Theorem~\ref{thm-isoEssTwstGrpdCstarAlg} where we show that this generalised evaluation map yields the desired isomorphism reconstructing the \(C^*\)-algebra \(B\) as a twisted groupoid \(C^*\)-algebra.
	
	\begin{proposition}\label{prop-commutingCondExps}
		Let $D$, $\Phi$, and $c$ be as in Lemma~\textup{\ref{lem-mapsToBAndCcGTwist}} and write \((G,\Sigma)\) for the Weyl twist \((G(A,B),\Sigma(A,B))\). 
		Identify $A$ with $C_0(G^{(0)})$ and \(\Mloc(A)\) with \(\bB(G^{(0)})/\mM(G^{(0)})\) as in (\ref{eq-mlocBorMeag}).
		There is a linear map $R:\Cc_c(G,\Sigma)\to\Mloc(C_0(G^{(0)}))$ that restricts functions to the unit space $G^{(0)}$ on $G$, such that the following diagram commutes:
		\begin{center}
			\begin{tikzcd}
				{D} \arrow[rr, "\Phi"] \arrow[d, "c"] &  & B \arrow[d, "E"]            \\
				{\Cc_c(G,\Sigma)} \arrow[rr, "R"]                                       &  & \bB(G^{(0)})/\mM(G^{(0)})
			\end{tikzcd}
		\end{center}
		\begin{proof}
			Let $f\in\Cc_c(U_n,\Sigma)$.
			Then $f$ has support contained in $U_n$, so $f|_{G^{(0)}}$ has support contained in $U_n\cap G^{(0)}$, which is equal to \(V_n\) by Lemma~\ref{lem-vnTrivPartUn}.
			Since $f$ is continuous and compactly supported on $U_n$, $f|_{G^{(0)}}$ is continuous and bounded on $V_n$, so defines a multiplier of $C_0(V_n)$.
			As $f|_{G^{(0)}}$ is zero outside of $V_n$, we see that $f|_{G^{(0)}}$ extends (by zero) to a multiplier on $C_0(V_n\cup(G^{(0)}\setminus \overline{V_n}))=C_0(V_n)\oplus C_0(V_n)^\perp$, which is an essential ideal in $C_0(G^{(0)})$.
			Thus $f|_{G^{(0)}}$ defines an element of $\Mloc(C_0(G^{(0)}))$.
			Since restriction to \(G^{(0)}\) is linear, we have a linear map \(R:\Cc_c(G(A,B),\Sigma(A,B))\to\Mloc(C_0(G^{(0)}))\) by defining \(R(f)\) as the class of \(f|_{G^{(0)}}\) in \(\Mloc(C_0(G^{(0)}))\).
			
			To show that \(E\circ\Phi=R\circ c\), it suffices (by linearity) to show that the two maps agree on the subspaces \(\Cc_c(U_n,\Sigma)\).
			Fix such a bisection \(U_n\) associated to a normaliser \(n\in N(A,B)\) and fix \(f\in\Cc_c(U_n,\Sigma)\).
			Write $f=\hat{n}h$ for some $h\in C_0(G^{(0)})$ using Lemma~\ref{lem-compactlySupportedSectionsAsEvalutatedNormalisers}.
			To show that $E(\Phi(f))=R(c(f))$ it suffices to show that both represent the same multiplier on any essential ideal of \(C_0(G^{(0)})\).
			In particular, it suffices to check on \(C_0(V_n)\oplus C_0(V_n)^\perp\).
			For $g\in C_0(V_n)$ we have 
			\[E(ng)g=(nh)|_{V_n}g=nhg=\hat n hg=f|_{V_n}g=R(f)g\]
			by Corollary~\ref{cor-explicitFuncRepExpOfNormaliser}.
			For $g^\perp\in C_0(V_n)^\perp$ we have $(nh)|_{V_n}g^\perp=0=f|_{V_n}g$ by Corollary~\ref{cor-explicitFuncRepExpOfNormaliser}, and $(nh)|_{G^{(0)}}g^\perp=0=f|_{G^{(0)}}$ since the supports of $n|_{G^{(0)}}$ and $g^\perp$ have zero intersection.
			Thus \(E(\Phi(f))\) and \(R(c(f))\) are both represented by the same multiplier on the essential ideal \(C_0(V_n)\oplus C_0(V_n)^\perp\) of $C_0(G^{(0)})$, whereby \(E(\Phi(f))=R(c(f))\).
		\end{proof}
	\end{proposition}

	The restriction map \(R\) in Proposition~\ref{prop-commutingCondExps} is the same as the conditional expectation $EL$ defined in \cite[Proposition~4.3]{KwasMeyer_EssCrossProd}, and so is the canonical local multiplier algebra-valued expectation giving rise to the essential groupoid $C^*$-algebra.
	
	\begin{corollary}\label{cor-mapDescendsToCc(Gtwist)}
		The map $\Phi$ descends via $c$ to a ${}^*$-homomorphism
		$$\Psi_0:\Cc_c(G(A,B),\Sigma(A,B))\to B$$ 
		satisfying $\Psi_0(c(f_n))=\Phi(f_n)$ for all $(f_n)\in D$, and  $R=E\circ\Psi_0$.
		\begin{proof}
			It suffices to show that the kernel of $c$ is contained in the kernel of $\Phi$.
			Contrapositively, we must show that if $\Phi(d)\neq 0$ for some $d\in D$ then $c(d)\neq 0$.
			Fix such $d\in D$ with $\Phi(d)\neq 0$.
			Since $c$ and $\Phi$ are ${}^*$-homomorphisms and $E$ is faithful, we have
			$$R(c(d)^*c(d))=R(c(d^*d))=E(\Phi(d^*d))=E(\Phi(d)^*\Phi(d))\neq 0.$$
			In particular, $c(d)^*c(d)$ cannot be zero since $R$ is linear, whereby $c(d)\neq 0$.
		\end{proof}
	\end{corollary}
	
	\begin{theorem}\label{thm-isoEssTwstGrpdCstarAlg}
		Let $A\subseteq B$ be an essential Cartan pair with faithful local conditional expectation $E:B\to\Mloc(A)$.
		Let $(G(A,B),\Sigma(A,B))$ by the Weyl twist associated to $A\subseteq B$.
		There is a ${}^*$-isomorphism $C^*_\ess(G(A,B),\Sigma(A,B))\cong B$ that restricts to an isomorphism $C_0(G(A,B)^{(0)})\cong A$ and intertwines $E$ with the canonical local expectation $C^*_\ess(G(A,B),\Sigma(A,B))\to\Mloc(C_0(G(A,B)^{(0)}))$, where $\Mloc(A)$ is identified with $\Mloc(C_0(G(A,B)^{(0)}))$.
		\begin{proof}
			Corollary~\ref{cor-mapDescendsToCc(Gtwist)} gives a ${}^*$-homomorphism $\Psi_0:\Cc_c(G(A,B),\Sigma(A,B))\to B$, and so this extends to a ${}^*$-homomorphism $\Psi:C^*(G(A,B),\Sigma(A,B))\to B$ by the universal property of the full twisted groupoid $C^*$-algebra.
			Since the canonical local conditional expectation $EL:C^*(G(A,B),\Sigma(A,B))\to\Mloc(A)$ is the continuous extension of the restriction map $R:\Cc_c(G(A,B),\Sigma(A,B))\to\Mloc(C_0(G(A,B)^{(0)}))$, Proposition~\ref{prop-commutingCondExps} implies that $\Psi$ intertwines $EL$ with $E$.
			Thus for $a\in C^*(G(A,B),\Sigma(A,B))$ we have $EL(a^*a)=E(\Psi(a)^*\Psi(a))=0$ if and only if $a\in\ker(\Psi)$ as $E$ is a faithful.
			In particular, the kernel of $\Psi$ is contained in the kernel of $EL$, so $\Psi$ descends to a homomorphism $\psi:C^*_\ess(G(A,B),\Sigma(A,B))\to B$.
			Moreover, $\psi$ is injective since it intertwines expectations, and both $EL$ and $E$ are faithful.
			The map $\psi$ is surjective since $\Phi$ has dense image in $B$ by Lemma~\ref{lem-mapsToBAndCcGTwist}.
		\end{proof}
	\end{theorem}
	
	\section{Uniqueness of the Weyl groupoid and twist}
	
	The Weyl groupoid and twist associated to an essential Cartan inclusion $A\subseteq B$ are not unique among the class of all \'etale groupoids giving rise to an isomorphic essential Cartan pair.
	The Weyl groupoid is however `final' in a certain sense, as we can show that any other twist over a groupoid giving the same essential Cartan pair will have the Weyl groupoid and twist as a quotient.
	We also show that the Weyl pair is unique among twists over effective groupoids.
	
	\begin{definition}\label{defn-twistHom}
		Let \((H,\Omega)\) and \((G,\Sigma)\) be twists over \'etale groupoids \(H\) and \(G\).
		A \emph{homomorphism of twists} \((\beta,\tilde\beta):(H,\Omega)\to (G,\Sigma)\) consists of groupoid homomorphisms \(\beta:H\to G\) and \(\tilde\beta:\Omega\to\Sigma\) such that the following diagram commutes:
		\begin{center}
			\begin{tikzcd}
				H^{(0)}\times\TT \arrow[d, "\beta|_{H^{(0)}}\times\Id"] \arrow[r] & \Omega \arrow[r] \arrow[d, "\tilde\beta"] & H \arrow[d, "\beta"] \\
				G^{(0)}\times\TT \arrow[r]                                        & \Sigma \arrow[r]                          & G                   
			\end{tikzcd}
		\end{center}
	\end{definition}

	\begin{theorem}\label{thm-groupoidSurjectionProperty}
		Let $A\subseteq B$ be an essential Cartan pair and let $(G,\Sigma)$ be the associated Weyl twist.
		Let $(H,\Omega)$ be a twist over an \'etale groupoid with locally compact Hausdorff unit space.
		Suppose there is an isomorphism $\varphi:C^*_\ess(H,\Omega)\to C^*_\ess(G,\Sigma)$ that restricts to an isomorphism $\varphi|_{C_0(H^{(0)})}:C_0(H^{(0)})\to C_0(G^{(0)})$ and intertwines local conditional expectations.
		Then there is a homomorphism of twists $(\beta_\varphi,\tilde\beta_\varphi):(H,\Omega)\to (G,\Sigma)$ that restricts to a homeomorphism of unit spaces, and  \(H^{(0)}\) is a dense subset of the kernel of \(\beta_\varphi\).
		Moreover $\tilde\beta_\varphi$ is a local homeomorphism.
	\end{theorem}

	To prove Theorem~\ref{thm-groupoidSurjectionProperty} we will need a few technical lemmata.
	Lemmata~\ref{lem-boundOnNormalisersGrpdAlg} and \ref{lem-bisectionSupBound} give bounds for the norms of some special elements of the essential \(C^*\)-algebra, and Lemma~\ref{lem-effectiveGrpdBisectionRecoveryProperty} is used to remove topological complications while retaining enough local information for our purposes.
	
	\begin{lemma}\label{lem-boundOnNormalisersGrpdAlg}
		Let $U,V\subseteq H$ be open bisections in a twisted groupoid $(H,\Omega)$.
		Suppose that $U\cap V=\emptyset$.
		Then for $f\in\Cc_c(U,\Omega)$ and $g\in\Cc_c(V,\Omega)$ we have $||f||\leq ||f+g||$ in the essential groupoid $C^*$-algebra.
		\begin{proof}
			The conditional expectation $E:C^*_\ess(H,\Omega)\to\Mloc(C_0(H^{(0)}))$ is contractive, and for sections supported on open bisections acts by restricting to the unit space.
			The functions $f^*g$ and $g^*f$ have supports contained in $U^{-1}V$ and $V^{-1}U$ respectively, which do not intersect the unit space $H^{(0)}$ since $U\cap V=\emptyset$, hence \(E(f^*g)=E(g^*f)=0\).
			Since $f$ and $g$ are supported on bisections, $f^*f$ and $g^*g$ are functions on the unit space, so $E(f^*f)=f^*f$ and $E(g^*g)=g^*g$.
			Moreover, $C_0(H^{(0)})$ is a commutative $C^*$-algebra so $||f^*f+g^*g||\geq ||f^*f||$.
			Thus we have
			\begin{align*}
				||f+g||^2&=||(f+g)^*(f+g)||\\
				&=||f^*f+f^*g+g^*f+g^*g||\\
				&\geq||E(f^*f+f^*g+g^*f+g^*g)||\\
				&=||E(f^*f)+E(g^*g)||\\
				&=||f^*f+g^*g||\\
				&\geq||f^*f||\\
				&=||f||^2.\qedhere
			\end{align*}
		\end{proof}
	\end{lemma}

	For the following lemma we remind the reader that the reduced \(C^*\)-algebra \(C^*_\red(H,\Omega)\) can be defined as a space of sections of the canonical line bundle associated to the twist \(\Omega\).
	As the essential \(C^*\)-algebra is a quotient of the reduced, one may choose genuine functions as representatives of elements of the essential \(C^*\)-algebra in the reduced \(C^*\)-algebra and compute using these.

	\begin{lemma}\label{lem-bisectionSupBound}
		Let $f\in C^*_\ess(H,\Omega)$ and suppose that $f$ is represented by a section $g\in C^*_\red(H,\Omega)$ of the line bundle.
		If $g$ has continuous restriction to some open bisection $U\subseteq H$ then
		\[\sup_{\gamma\in U}|g(\gamma)|\leq||f||.\]
		\begin{proof}
			Since $g$ is continuous on $U$ the supremum $\sup_{\gamma\in U}|g(\gamma)|$ is equal to $\sup_{\gamma\in C}|g(\gamma)|$ for any comeagre subset $C\subseteq U$.
			Thus taking an infimum over comeagre subsets of $U$ changes nothing, and we see 
			$$\sup_{\gamma\in U}|g(\gamma)|^2=\inf_{C\subseteq U}\sup_{\gamma\in C}|g(\gamma)|^2\leq\inf_{D\subseteq H^{(0)}}\sup_{x\in D}|g^*g(x)|= ||EL(g^*g)||=||EL(f^*f)||\leq||f||^2,$$
			where $D\subseteq H^{(0)}$ ranges over all comeagre subsets of $H^{(0)}$.
		\end{proof}
	\end{lemma}
	
	\begin{lemma}\label{lem-effectiveGrpdBisectionRecoveryProperty}
		Let $G$ be an effective groupoid with locally compact Hausdorff unit space.
		Let $U\subseteq G$ be an open bisection, and let $\overline{U}^\circ$ be the interior of the closure of $U$.
		Then $U=\overline{U}^\circ\cdot s(U)=r(U)\cdot\overline{U}^\circ$.
		\begin{proof}
			The inclusion $U\subseteq\overline{U}^\circ\cdot s(U)$ is clear.
			Fix $\gamma\in\overline{U}^\circ\cdot s(U)$.
			Then there exists $\eta\in U$ with $s(\eta)=s(\gamma)$, and we have $\gamma\eta^{-1}\in\overline{U}^\circ\cdot s(U)\cdot U^{-1}=\overline{U}^\circ U^{-1}$.
			We claim this is contained in \(\overline{UU^{-1}}\).
			Fix a net \((\gamma_\lambda)\) contained in \(U\) which converges to \(\gamma\).
			Then the image \((s(\gamma_\lambda))\) under the source map is a net converging to \(s(\gamma)=s(\eta)\), and since \(s\) restricts to a homeomorphism \(U\to s(U)\) we see that \((\gamma_\lambda)\) also converges to \(s|_U^{-1}(s(\eta))=\eta\).
			Hence the net \((\gamma_\lambda\gamma_\lambda^{-1})\) converges to \(\gamma\eta^{-1}\), so \(\gamma\eta^{-1}\) is contained in the closure of \(UU^{-1}\), which is itself contained in the closure of the unit space of \(G\).
			Since the multiplication map in an \'etale groupoid is open (see e.g. \cite[Lemma~2.4.11]{Sims_EtaleGrpds}), we further see that \(\overline{U}^\circ U^{-1}\) is contained in \(\overline{G^{(0)}}^\circ\).
			Lastly, as $G$ is effective and the closure of the unit space of a groupoid with Hausdorff unit space consists of isotropy by Lemma~\ref{lem-HDUnitsCloseInIsot}, we see that $\overline{G^{(0)}}^\circ=G^{(0)}$ and so $\gamma\eta^{-1}\in G^{(0)}$, and thus $\gamma=\eta\in U$, and so $\overline{U}^\circ\cdot s(U)= U$.
			A similar argument shows $r(U)\cdot\overline{U}^\circ= U$.
		\end{proof}
	\end{lemma}

	\begin{proof}[Proof of Theorem~\textup{\ref{thm-groupoidSurjectionProperty}}]
		First we shall establish that a map $\tilde\beta_\varphi:\Omega\to \Sigma$ exists.	
		Fix $\sigma\in\Omega$ and let $U\subseteq H$ be a bisection with $\sigma\in\Omega|_U$.
		Let $\theta_\varphi$ be the inverse to \mbox{$\varphi|_{H^{(0)}}^*:G^{(0)}\to H^{(0)}$}; the homeomorphism induced by considering the restricted isomorphism of commutative $C^*$-algebras $\varphi|_{H^{(0)}}:C_0(H^{(0)})\to C_0(G^{(0)})$.
		Viewing elements of $\Cc_c(U,\Omega)$ as continuous functions $\Omega|_U\to\CC$ satisfying $f(\tau z)=f(\tau)\bar{z}$ for $z\in\TT$, let $f_\sigma:\Omega|_U\to\CC$ be such a continuous compactly supported function with $f(\sigma)=1$.
		Then $\varphi(f_\sigma)$ is a normaliser of $C_0(G^{(0)})\subseteq C^*_\ess(G,\Sigma)$, and $\varphi(f_\sigma^*f_\sigma)(\theta_\varphi(s(\sigma)))=(f_\sigma^*f_\sigma)(s(\sigma))=|f_\sigma(\sigma)|^2=1$, so $\theta_\varphi(s(\sigma))\in\dom(\varphi(f_\sigma))$.
		Thus $\tilde\beta_{\varphi}(\sigma,U, f_\sigma):=[\varphi(f_\sigma),\theta_\varphi(s(\sigma))]$ is an element of the Weyl twist $\Sigma$.
		
		We claim that $\tilde\beta_\varphi(\sigma, U,f_\sigma)$ does not depend on the choice of $U\subseteq H$ or $f_\sigma\in\Cc_c(U,\Sigma)$.
		To see this, fix another bisection $U'\subseteq H$ with $\sigma\in\Omega|_{U'}$ and $g_\sigma\in\Cc_c(U',\Omega)$ with $g_\sigma(\sigma)=1$.
		
		Then $U\cap U'$ is non-empty since \(\sigma\) is contained in a fibre over \(\Omega|_{U\cap U'}\), and there exists $h\in C_0(X)$ with $\supp(h)\subseteq s(U\cap U')$ and $h|_K=1$ on a compact neighbourhood $K$ of $s(\sigma)$.
		The functions $f_\sigma h$ and $g_\sigma h$ are both compactly supported sections with support contained in $U\cap U'$.
		Thus there are functions $a,a'\in C_0(X)$ with $a(s(\sigma))=1=a'(s(\sigma))$ and $f_\sigma h a=g_\sigma h a'$.
		We then have $\varphi(f_\sigma)\varphi(ha)=\varphi(g_\sigma)\varphi(ha')$ and $\varphi(ha)(\theta_\varphi(s(\sigma)))=ha(s(\sigma))=1=ha'(s(\sigma))=\varphi(ha')(\theta_\varphi(s(\sigma)))$, so $[\varphi(f_\sigma),\theta_\varphi(s(\sigma))]=[\varphi(g_\sigma),\theta_\varphi(s(\sigma))]$ and $\tilde\beta_\varphi$ depends only on $\sigma\in\Omega$.
		We shall now simply write \(\tilde\beta(\sigma)\) for the element \(\tilde\beta(\sigma,U,f_\sigma)\), which defines our function \(\tilde\beta\colon\Omega\to\Sigma\).
		
		To see that $\tilde\beta_\varphi$ is a groupoid homomorphism, fix a composable pair $(\sigma,\tau)\in\Omega^{(2)}$.
		For all $a\in C_0(G^{(0)})$ we have
		\begin{align*}
			a(r(\tilde\beta_\varphi(\tau)))&=a(\alpha_{\varphi(f_\tau)}(\theta_\varphi(s(\tau))))\\
			&=a(\alpha_{\varphi(f_\tau)}(\theta_\varphi(s(\tau))))\varphi(f_\tau^*f_\tau)(\theta_\varphi(s(\tau)))\\
			&=(\varphi(f_\tau)^*a\varphi(f_\tau))(\theta_\varphi(s(\tau)))\\
			&=f_\tau^*\varphi^{-1}(a)f_\tau(s(\tau))\\
			&=\varphi^{-1}(a)(r(\tau))f_\tau^*f_\tau(s(\tau))\\
			&=a(\theta_\varphi(r(\tau)))\\
			&=a(\theta_\varphi(s(\sigma)))\\
			&=a(s(\tilde\beta_\varphi(\sigma))).
		\end{align*}
		Since this holds for all $a\in C_0(G^{(0)})$, we see that $r(\tilde\beta_\varphi(\tau))=s(\tilde\beta_\varphi(\sigma))$, so the pair $(\tilde\beta_\varphi(\sigma),\tilde\beta_\varphi(\tau))$ is composable.
		
		Let $U,V\subseteq H$ be bisections with $\sigma\in\Omega|_U$ and $\tau\in\Omega|_V$, and let $f_\sigma\in\Cc_c(U,\Omega)$ and $f_\tau\in\Cc_c(V,\Omega)$ be functions with $f_\tau(\sigma)=f_\tau(\tau)=1$.
		The product $f_\sigma f_\tau$ belongs to $\Cc_c(UV,\Omega)$, and since both functions are supported on bisections we have $(f_\sigma f_\tau)(\sigma\tau)=f_\sigma(\sigma)f_\tau(\tau)=1$, whereby
		\begin{align*}
			\tilde\beta_\varphi(\sigma\tau)&=[\varphi(f_\sigma f_\tau),\theta_\varphi(s(\sigma\tau))]\\
			&=[\varphi(f_\sigma),\theta_\varphi(s(\sigma))][\varphi(f_\tau),\theta_\varphi(s(\tau))]\\
			&=\tilde\beta_\varphi(\sigma)\tilde\beta_\varphi(\tau).
		\end{align*}			
		We now show that $\tilde\beta_\varphi$ is a local homeomorphism.
		The topology on $\Sigma$ is generated by open sets specified by normalisers $n\in N(A,B)$ and the homeomorphisms 
		$$G_n:\TT\times\dom(n)\to\Sigma|_{U_n},\qquad G_n(t,x)=[tn,x].$$
		For $\sigma\in\Omega$ and an open bisection $U\subseteq H$ with $\sigma\in\Omega|_U$, pick $f_\sigma\in\Cc_c(U,\Omega)$ with $f_\sigma(\sigma)=1$.
		Let $U_\sigma\subseteq\Omega$ be the open support of $f_\sigma$.
		Define the map $H_\sigma:\TT\times s(U_\sigma)\to U_\sigma$ by $H_\sigma(z,x)=z\tau_x$, where $\tau_x\in U_\sigma$ is the unique element of $U_\sigma$ with $s(\tau_x)=x$ and $f_\sigma(\tau_x)>0$.
		Then $H_\sigma$ is a homeomorphism since $f_\sigma$ is continuous and non-zero on $U_\sigma$.
		Note now that $f_\sigma(z\tau_x)=\bar{z}f_\sigma(\tau_x)$, so $zf_\sigma(z\tau_x)>0$ for any $\tau_x\in U_\sigma$ with $f_\sigma(\tau_x)>0$, giving $\tilde\beta(z\tau_x)=[z\varphi(f_\sigma),\theta_\varphi(s(\tau_x))]$.
		Letting $n_\sigma:=\varphi(f_\sigma)$, for $(z,x)\in \TT\times s(U_\sigma)$  we compute
		\begin{align*}
			G_{n_\sigma}^{-1}\circ\tilde\beta_\varphi\circ H_{\sigma}(z,x)&=G_{n_\sigma}^{-1}\circ\tilde\beta_\varphi(z\tau_x)\\
			&=G_{n_\sigma}^{-1}[z\varphi(f_\sigma),\theta_\varphi(s(\tau_x))]\\
			&=G_{n_\sigma}^{-1}[zn_\sigma,\theta_\varphi(s(\tau_x))]\\
			&=(z,\theta_\varphi(s(\tau_x))).
		\end{align*} 
		Thus the composition $G_{n_\sigma}^{-1}\circ\tilde\beta_\varphi\circ H_{\sigma}$ agrees with the homeomorphism $\Id_\TT\times\theta_\varphi$ on its domain, implying that $\tilde\beta_\varphi$ is a homeomorphism on $U_\sigma$.
		Hence $\tilde\beta_\varphi$ is a local homeomorphism.
		This argument also shows that $\tilde\beta_\varphi$ descends to a local homeomorphism $\beta_\varphi:H\to G$, as one need only ignore the $\TT$-component in the homeomorphisms $H_\sigma$ and $G_{n_\sigma}$.
		Moreover, since \(H_\sigma\) and \(G_{n_\sigma}\) commute with the \(\TT\)-actions, we also see \(\tilde\beta_\varphi(z\sigma)=z\tilde\beta_\varphi(\sigma)\) for all \(\sigma\in\Omega\) and hence the first square in Definition~\ref{defn-twistHom} commutes.
		Hence \((\beta_\varphi,\tilde\beta_\varphi)\) forms a homomorphism of twists.\\
		
		To see that each \(\tilde\beta_\varphi\) and \(\beta_\varphi\) are surjective, it suffices to show that \(\tilde\beta_\varphi\) is.
		To this end, fix \([n,x]\in\Sigma\), and let \(U\subseteq G\) be an open bisection neighbourhood of \([\alpha_n,x]\).
		Pick \(g\in\Cc_c(U,\Sigma)\) such that there is a neighbourhood \(W\subseteq U\) of \([n,x]\) with \(|g(\nu)|=1\) for all \(\nu\in\Sigma|_W\) (viewing \(g\) as a function \(\Sigma\to\CC\)).
		Since \(\varphi\) is surjective and the image of \(\Cc_c(H,\Omega)\) in \(C^*_\ess(H,\Omega)\) is dense, there are open bisections \(U_1,\dots,U_n\subseteq H\) and elements \(f_i\in\Cc_c(U_i,\Omega)\) such that 
		\[\left|\left|\sum_{i=1}^n \varphi(f_i)-g\right|\right|<\frac12.\]
		Each \(f_i\) is a normaliser of \(C_0(H^{(0)})\) in \(C^*_\ess(H,\Omega)\), and hence the \(\varphi(f_i)\) normaliser \(C_0(G^{(0)})\) since \(\varphi\) restricts to an isomorphism of these subalgebras.
		Moreover, since the \(f_i\) can be picked to have compact support, the images \(\varphi(f_i)^*\varphi(f_i)\) of \(f_i^*f_i\) belong to \(C_c(G^{(0)})\), so the \(\varphi(f_i)\) are normalisers of \(C_0(G^{(0)})\) in \(C^*_\ess(G,\Sigma)\) with compact support.
		Hence we may apply Lemma~\ref{lem-localEvaluationMap}, yielding functions \(\widehat{\varphi(f_i)}\in\Cc_c(U_{\varphi(f_i)},\Sigma)\) supported on the sets \(U_{\varphi(f_i)}\) as in (\ref{eq-UnBisections}), which are then mapped to \(\varphi(f_i)\) in \(C^*_\ess(G,\Sigma)\) by the map in Lemma~\ref{lem-mapsToBAndCcGTwist}.
		We claim that at least one \(f_i\) satisfies \(\widehat{\varphi(f_i)}[n,x]\neq 0\).
		Suppose, for a contradiction, that \(\widehat{\varphi(f_i)}[n,x]=0\) for all \(i=1,\dots,n\).
		By \cite[Lemma~7.13]{KwasMeyer_EssCrossProd}, there is a comeagre subset \(C\subseteq W\) on which each of the \(\widehat{\varphi(f_i)}\) are continuous, and hence there is an open neighbourhood \(W'\subseteq W\) of \([\alpha_n,x]\) such that 
		\[|\widehat{\varphi(f_i)}[m,y]|<\frac1{2n}\]
		for each \(i=1,\dots,n\) and all \([m,y]\in\Sigma|_{W'\cap C}\).
		The section \(g^*\widehat{\varphi(f_i)}\) is a function supported on \(U^*U_i\), and when restricted further to \(U^*W\) this function takes values
		\[\left(g^*\widehat{\varphi(f_i)}\right)(y)=\overline{g(\eta_y)}\widehat{\varphi(f_i)}(\eta_y)\]
		where \(\eta_y\) denotes the unique element of \(W\) with \(s(\eta_y)=y\).
		Taking the norm of this yields
		\[\left|\left(g^*\widehat{\varphi(f_i)}\right)(y)\right|=\left|\widehat{\varphi(f_i)}(\eta_y)\right|,\]
		as \(\eta_y\in W\) implies \(|g(\eta_y)|=1\) by choice of \(g\).
		We now analyse
		\begin{align*}
			\left|\left|\sum_{i=1}^n\varphi(f_i)-g\right|\right|&\geq \left|\left|\sum_{i=1}^n g^*\varphi(f_i)-g^*g\right|\right|\\
			&\geq \left|\left|EL\left(\sum_{i=1}^n g^*\varphi(f_i)\right)-g^*g\right|\right|\\
			&=\sup_{\substack{D\subseteq G^{(0)}\\\text{comeagre}}}\sup_{y\in D}\left|\sum_{i=1}^n \left(g^*\widehat{\varphi(f_i)}\right)(y)-g^*g(y)\right|\\
			&\geq \sup_{\substack{D\subseteq s(W')\\\text{comeagre}}}\sup_{y\in D}\left|\sum_{i=1}^n \left(g^*\widehat{\varphi(f_i)}\right)(y)-g^*g(y)\right|\\
			&\geq\sup_{\substack{D\subseteq s(W')\\\text{comeagre}}}\sup_{y\in D}\left|\left|\sum_{i=1}^n \left(g^*\widehat{\varphi(f_i)}\right)(y)\right|-\left|g^*g(y)\right|\right|\\
			&\geq\sup_{y\in s(C)}\left|\sum_{i=1}^n \widehat{\varphi(f_i)}(\eta_y)-1\right|\\
			&>1-n\frac1{2n}\\
			&=\frac12,
		\end{align*}
		contradicting \(\left|\left|\sum_{i=1}^n\varphi(f_i)-g\right|\right|<\frac12\).
		Hence \(\widehat{\varphi(f_i)}[n,x]\) must be non-zero for some \(i\).
		The formula in Lemma~\ref{lem-localEvaluationMap} implies \(x\) belongs to the domains of the partial homeomorphisms \(\alpha_n\) and \(\alpha_{\varphi(f_i)}\), and moreover that the germs of these partial homeomorphisms agree at \(x\).
		This implies \([\alpha_{\varphi(f_i)},x]=[\alpha_n,x]\), whereby \([z\varphi(f_i),x]=[n,x]\) for some \(z\in\TT\).
		Moreover, \(f_i^*f_i(\theta_\varphi(x))=\varphi(f_i^*f_i)(x)\neq 0\) so there is a unique element \(\sigma\in\Omega\) in the support of \(f_i\) with \(\theta_\varphi(s(\sigma))=x\) and \(zf_i(\sigma)=1\).
		We then have \(\tilde\beta_\varphi(\sigma)=[\varphi(zf_i),\theta_\varphi(s(\sigma))]=[n,x]\), and thus \(\tilde\beta_\varphi\) is surjective.\\

		Finally we shall show that the (twist over the) unit space $H^{(0)}$ is dense in the kernel of $\tilde\beta_\varphi$.
		The kernel of $\beta_\varphi$, that is, the preimage of the unit space $G^{(0)}$ under $\beta_\varphi$ is open since $G$ is \'etale.
		Let $U\subseteq\ker(\beta_\varphi)$ be an open bisection and suppose that $U\cap H^{(0)}=\emptyset$.
		Then for any $f\in\Cc_c(U,\Omega)$ we have $[\varphi(f),\theta_\varphi(x)]\in G^{(0)}\times\TT$ for any $x\in\dom(f)$, so $\varphi(f)\in C_0(G^{(0)})$.
		Since $\varphi$ restricts to an isomorphism of subalgebras $C_0(H^{(0)})\cong C_0(G^{(0)})$, the function $f$ must represent an element of $C_0(H^{(0)})$ in $C^*_\ess(H,\Omega)$.
		So there is a function $g\in C_0(H^{(0)})$ with \(f=g\) in $C^*_\ess(H,\Omega)$.
		Consider $f$ and $g$ as sections of the line bundle associated to $(H,\Omega)$. 
		Since $U\cap H^{(0)}=\emptyset$, Lemma~\ref{lem-boundOnNormalisersGrpdAlg} gives $||f||\leq ||f-g||=0$, implying that $f=0$.
		Thus $U$ must be empty, so every non-empty open subset of $\ker(\beta_\varphi)$ has non-empty intersection with $H^{(0)}$; equivalently $H^{(0)}$ is dense in $\ker(\beta_\varphi)$.
	\end{proof}
	
	\begin{corollary}\label{cor-WeylGrpdUniqueAmongEffectiveGrpds}
		Let $A\subseteq B$ be an essential Cartan pair and let $(G,\Sigma)$ be the associated Weyl twist.
		Let $(H,\Omega)$ be a twist over an \'etale groupoid with locally compact Hausdorff unit space.
		Suppose there is an isomorphism $\varphi:C^*_\ess(H,\Omega)\to C^*_\ess(G,\Sigma)$ that intertwines conditional expectations and restricts to an isomorphism $\varphi|_{C_0(H^{(0)})}:C_0(H^{(0)})\to C_0(G^{(0)})$.
		If $H$ is effective then the homomorphism $(\beta_\varphi,\tilde\beta_\varphi):(H,\Omega)\to(G,\Sigma)$ in Theorem~\textup{\ref{thm-groupoidSurjectionProperty}} is an isomorphism of twists.
		\begin{proof}
			The kernel of $\beta_\varphi$ is an open normal isotropy group bundle in $H$ by Theorem~\ref{thm-groupoidSurjectionProperty}, and so is contained in the interior of the isotropy of $H$.
			Since $H$ is effective, the interior of the isotropy is exactly the unit space of $H$ and so the map $\beta_\varphi$ is injective, hence an isomorphism between $H$ and $G$.
			
			Suppose $\tilde\beta_\varphi(\omega)\in\Sigma^{(0)}$ for some $\omega\in\Omega$.
			Let $q_\Omega:\Omega\to H$ and $q_\Sigma:\Sigma\to G$ be the canonical quotient maps associated to the twists. Since $(\beta_\varphi,\tilde\beta_\varphi)$ is a homomorphism of twists, we have $\beta_\varphi(q_\Omega(\omega))=q_\Sigma(\tilde\beta_\varphi(\omega))\in G^{(0)}$, so $q_\Omega(\omega)\in H^{(0)}$ as $\beta_\varphi$ is an isomorphism.
			Hence $\omega\in H^{(0)}\times\TT$, so $\omega=(s(\omega),z)$ for some $z\in\TT$.

			Hence we have
			\begin{align*}
				\tilde\beta_\varphi(\omega)&=\tilde\beta_\varphi(z\cdot s(\omega))\\
				&=z\tilde\beta_\varphi(s(\omega))\\
				&=z\cdot s(\tilde\beta_\varphi(\omega))\\
				&=z\cdot \tilde\beta_\varphi(\omega),
			\end{align*}
			where the second equality holds since \(\tilde\beta_\varphi\) comes from a homomorphism of twists by Theorem~\ref{thm-groupoidSurjectionProperty}.
			This then yields \(z=1\), and hence $\tilde\beta_\varphi$ and \(\beta_\varphi\) are injective, so yield an isomorphism of twists.
		\end{proof}
	\end{corollary}

	We thank an anonymous referee for providing a streamlined argument showing the injectivity of \(\tilde\beta\) in Corollary~\ref{cor-WeylGrpdUniqueAmongEffectiveGrpds}.

	We have seen in Lemma~\ref{lem-weylGrpdIsEffective} that the Weyl groupoid associated to an essential Cartan pair is effective.
	We now show that twists over effective groupoids give rise to essential Cartan pairs; this amounts to showing that \(C_0(G^{(0)})\) is a maximal commutative subalgebra of \(C^*_\ess(G,\Sigma)\) if \(G\) is effective.
	For Hausdorff groupoids this is a fairly straightforward proof, but typically relies on being to evaluate sections in the reduced \(C^*\)-algebra at individual points.
	Instead, we must show that a section which commutes with the subalgebra of functions on the unit space must have a comeagre amount of its support contained in the unit space, and moreover it can be represented continuously there. 
	
	\begin{proposition}[cf. {\cite[2.4.7]{Renault_GrpdApproach}}]\label{prop-effectiveGrpdGivesMasa}
		Let \(G\) be an effective groupoid and let \(\Sigma\) be a twist over \(G\).
		The inclusion \(C_0(G^{(0)})\subseteq C^*_\ess(G,\Sigma)\) is maximal commutative.
		\begin{proof}
			Fix \(f\in C^*_\ess(G,\Sigma)\) commuting with \(C_0(G^{(0)})\) and pick \(\tilde{f}\in C^*_\red(G,\Sigma)\) with image equal to \(f\) in the essential quotient.
			
			Applying \cite[Lemma~7.13]{KwasMeyer_EssCrossProd} we see that there is a comeagre subset \(C\subseteq G\) on which \(\tilde{f}\) is continuous.
			If the support of \(\tilde{f}\) does not intersect \(C\) then \(f=0\) by \cite[Proposition~7.18]{KwasMeyer_EssCrossProd}.
			Else, for \(\gamma\in C\) with \(\tilde{f}(\gamma)\neq 0\), there is \(\eps>0\) and a neighbourhood bisection \(U\subseteq G\) of \(\gamma\) with \(|f(\gamma)|>\eps\) for all \(\gamma\in U\cap C\).
			Then, for all \(\eta\in U\cap C\) and all \(g\in C_0(G^{(0)})\) we have
			\begin{align*}
				g(r(\eta))\tilde{f}(\eta)&=(g\tilde{f})(\eta)\\
				&=(\tilde{f}g)(\eta)\\
				&=\tilde{f}(\eta)g(s(\eta)).
			\end{align*}
			Since \(\tilde{f}(\eta)\) is non-zero for all \(\eta\in U\cap C\) and functions \(g\in C_0(G^{(0)})\) separate points in \(G^{(0)}\), we see that \(r(\eta)=s(\eta)\).
			Hence \(U\cap C\) is contained in the isotropy of \(G\), and so is \(U\) since \(C\subseteq G\) is comeagre (hence dense).
			Then \(U\) is a bisection contained in the interior of the isotropy of \(G\), therefore \(U\subseteq G^{(0)}\) since \(G\) is effective.
			Thus \(\tilde{f}\) is continuous on a comeagre subset of \(G^{(0)}\) and (up to a comeagre subset) is supported there, so \(f\in C_0(G^{(0)})\).
		\end{proof}
	\end{proposition}

	\begin{remark}
		Exel and Pitts define their essential \(C^*\)-algebra of a twist \((G,\Sigma)\) as the quotient of \(C^*_\red(G,\Sigma)\) by the \emph{grey ideal}.
		If \(G\) is topologically principal i.e. the subset of \(G^{(0)}\) of points with non-trivial isotropy fibres is dense, Kwa\'sniewski and Meyer showed that this construction yields an isomorphic \(C^*\)-algebra to their own construction when the groupoid \cite[Section~7]{KwasMeyer_EssCrossProd}.
		However, should \(G\) fail to be topologically principal, the grey ideal is much larger, and there are examples of effective groupoids for which the essential \(C^*\)-algebra collapses to zero entirely.

		Exel and Pitts define \emph{weak Cartan inclusions} \cite[Definition~2.11.5]{ExelPitts_AlgNonHDGrpds} and show that any such inclusion is isomorphic to one arising from a twist over a topologically principal groupoid with their essential twisted groupoid \(C^*\)-algebra \cite[Corollary~3.9.5]{ExelPitts_AlgNonHDGrpds}.
		In their setup, only separable \(C^*\)-algebras and second-countable groupoids are considered, and effective second-countable groupoids are automatically topologically principal, so the two constructions agree in this context.
		The definitions of weak Cartan inclusions and essential Cartan inclusions differ outside of the second-countable setting, the former corresponding to topologically principal groupoids, and the latter corresponding to effective groupoids.
	\end{remark}

	\subsection{Non-uniqueness for non-effective groupoids}
	
	One may ask the question of when two \'etale groupoids give rise to isomorphic essential groupoid $C^*$-algebras.
	We are not able to answer this question in full generality, but we can provide this for the condition when one groupoid is contained in another as an open subgroupoid.
	We answer this question in the case where the twist over the groupoid is trivial.

	As noted in Remark~\ref{rem-singularIdealMeagreSupp}, if \(G\) is an \'etale groupoid and \(H\) is an open subgroupoid then there is an inclusion \(\Cc_c(H)\hookrightarrow\Cc_c(G)\) given by extending functions on \(H\) by zero to \(G\).
	This map is a \({}^*\)-homomorphism, and descends to a homomorphism \(C^*_\ess(H)\to C^*_\ess(G)\).
	This homomorphism is injective on the essential groupoid \(C^*\)-algebras since it intertwines expectations and the expectations are faithful on the essential quotients.
	Hence, we may always consider \(C^*_\ess(H)\subseteq C^*_\ess(G)\) for an open subgroupoid \(H\subseteq G\).
	It is possible that the resulting inclusion is an equality, even if the underlying subgroupoid \(H\subseteq G\) is not trivially all of \(G\), but contains `essentially enough' of the information of \(G\).
	The notion of `essentially enough' can be made rigorous, but to do so we require a technical point-set topological tool.

	\begin{lemma}\label{lem-weirdTopoResult}
		Let $X$ be a topological space.
		Let $x\in X$ be a point and let $U\subseteq X$ be a neighbourhood of $x$.
		Let $V_1,\dots,V_\ell\subseteq X$ be open sets such that $x\notin\overline{V_k}^\circ$ for $1\leq k\leq\ell$.
		That is, $x$ does not lie in the interior of the closure of any of the $V_k$.
		Then $Z:=U\setminus\bigcup_{k=1}^\ell \overline{V_k}$ is a non-empty open set, and $x$ lies in the closure of $Z$.
		\begin{proof}
			Since $Z$ is the intersection of $U$ with the complements of finitely many open sets, we see that $Z$ is open.
			To see that $Z$ is non-empty, we first note that $Z$ can be expressed as
			$$Z=U\setminus\bigcup_{k=1}^\ell \overline{V_k}=\left(U\setminus\bigcup_{k=1}^\ell \partial(\overline{V_k}^\circ)\right)\cap\left(U\setminus\bigcup_{k=1}^\ell \overline{V_k}^\circ\right).$$
			Since each $\partial(\overline{V_k}^\circ)$ is the boundary of an open set, it is nowhere dense, and the finite union \(\bigcup_{k=1}^\ell\partial(\overline{V_k}^\circ)\) is again nowhere dense.
			Thus $Z$ is the complement of a nowhere dense subset of $U\setminus\bigcup_{k=1}^\ell \overline{V_k}^\circ$, and hence is dense.
			The set $U\setminus\bigcup_{k=1}^\ell \overline{V_k}^\circ$ is non-empty since $x$ is not contained in any $\overline{V_k}^\circ$ by hypothesis, so $x$ lies in the closure of $Z$ and in particular $Z$ is non-empty.
		\end{proof}
	\end{lemma}
	
	\begin{proposition}\label{prop-essGrpdCstarAlgEqualCond}
		Let $H\subseteq G$ be an open subgroupoid of an \'etale groupoid with locally compact Hausdorff unit space.
		Consider $C^*_\ess(H)\subseteq C^*_\ess(G)$.
		Then $C^*_\ess(H)=C^*_\ess(G)$ if and only if for every $\gamma\in G$ there exists a bisection $V\subseteq H$ with $\gamma\in \overline{V}^\circ\cdot s(V)$.
		In particular, $H$ is dense in $G$ and $H^{(0)}=G^{(0)}$.
	\end{proposition}
	In the criterion described in Proposition~\ref{prop-essGrpdCstarAlgEqualCond}, the closure of \(V\) is taken in \(G\) rather than \(H\), otherwise the described condition would reduce to \(H=G\) and the equivalence would fail (see Example~\ref{eg-grpdsWithSameEssCstarAlg}).
	\begin{proof}
		Suppose first that for every $\gamma\in G$ there is some bisection $V\subseteq H$ with $\gamma\in\overline{V}^\circ\cdot s(V)$.
		Let $U\subseteq G$ be a bisection neighbourhood of $\gamma$ contained in $\overline{V}^\circ\cdot s(V)$; such a bisection exists since $\overline{V}^\circ\cdot s(V)$ is open.
		We claim that $\Cc_c(U)$ and $\Cc_c(V\cdot s(U))$ are identified in the essential groupoid $C^*$-algebra.
		To show this, fix $f\in\Cc_c(U)$.
		Note that $s|_U^{-1}\circ s:V\cdot s(U)\to U$ is a homeomorphism, so we can define $\tilde{f}:=f\circ s|_{U}^{-1}\circ s$ on $V\cdot s(U)$, and extend by zero to a function on $H$.
		Then $\tilde{f}$ is a continuous and compactly supported function on $V$ since $f$ is, and $s|_U^{-1}\circ s$ is a homeomorphism between $U$ and $V\cdot s(U)$.
		
		We claim $f-\tilde{f}$ has meagre support, which by \cite[Proposition~7.18]{KwasMeyer_EssCrossProd} shows that $f$ and $\tilde{f}$ represent the same element of $C^*_\ess(G)$.
		Since $s|_U^{-1}\circ s$ restricts to the identity on $U\cap V$, the set-theoretic support of $f-\tilde{f}$ (that is, the set of point where \(f-\tilde f\) is non-zero) must be contained in $(\supp^\circ(f)\setminus (U\cap V))\cup(\supp^\circ(\tilde{f})\setminus (U\cap V))$.
		Since the open supports of $f$ and $\tilde{f}$ are contained in $U$ and $V$ respectively, this reduces to $\supp^\circ(f)\setminus V\cup\supp^\circ(\tilde{f})\setminus U$.
		We shall show that each of these components of the union is meagre.
		
		Note that $\supp^\circ(f)\subseteq U\subseteq\overline{V}^\circ\cdot s(U)\subseteq \overline{V}$, and $\overline{V}\setminus V$ is meagre as it is closed and contains no open subsets.
		Hence $\supp^\circ(f)\setminus V$ is meagre.
		
		We claim that $U$ is dense in the support of $\tilde{f}$. 
		Then by a similar argument as above the set $\supp^\circ(\tilde{f})\subseteq \overline{U}\setminus U$ is meagre.
		Fix an open subset $W\subseteq\supp^\circ(\tilde{f})$.
		Note that $W$ is open in $G$ since $\supp^\circ(\tilde{f})$ is.
		Suppose $W\cap U=\emptyset$.
		We shall show that $W$ is then empty, whereby $\supp^\circ(\tilde{f})$ is contained in the closure of $U$.
		Note that since $s(W)\subseteq s(\supp^\circ(\tilde{f}))\subseteq s(U)$, it suffices to show that $U\cdot s(W)$ is empty.
		We then note that $(U\cdot s(W))\cap V=U\cap V\cdot s(W)=U\cap W=\emptyset$, whereby $U\cdot s(W)$ is empty since $U$ lies in the closure of $V$, so every non-empty open subset of $U$ has non-empty intersection with $V$.
		Thus $W\cap U\neq\emptyset$ for all non-empty open subsets $W\subseteq\supp^\circ(\tilde{f})$, whereby $\supp^\circ(\tilde{f})$ is contained in the closure of $U$ so $\supp^\circ(\tilde{f})\setminus U$ is meagre.
		It follows that $\supp^\circ(f-\tilde{f})$ is a meagre set and the homeomorphism $s|_V^{-1}\circ s$ induces an identification of $\Cc_c(U)$ and $\Cc_c(V)$ in $C^*_\ess(G)$.
		We can do this on neighbourhoods around any point in $G$, so a partition of unity argument shows that $C^*_\ess(H)$ is dense in $C^*_\ess(G)$, whereby they are equal.\\
		
		We shall show the converse via contrapositive.
		Suppose there exists $\gamma\in G\setminus H$ with the property that, for any bisection $V\subseteq H$ we have $\gamma\notin \overline{V}^\circ\cdot s(V)$.
		That is, for any bisection $V\subseteq H$ either no open neighbourhood of $\gamma$ is contained in $\overline{V}$, or $s(\gamma)\notin s(V)$.
		Let $U$ be a bisection containing $\gamma$ and let $K\subseteq U$ be a compact neighbourhood also containing $\gamma$, and denote the interior of $K$ by $K^\circ$.
		Let $f\in\Cc_c(U)$ be a function with $f|_K=1$.
		We claim that the image of $f$ in \(C^*_\ess(G)\) does not belong to $C^*_\ess(H)$.
		Recall that the subspaces $\Cc_c(V)$ for bisections $V\subseteq H$ span a dense subalgebra of $C^*_\ess(H)$, so it suffices to show that $f$ cannot be approximated by any finite sum of elements belonging to subspaces of this form.
		Fix bisections $V_i\subseteq H$ and functions $g_i\in\Cc_c(V_i)$ for $i=1,\dots,n$.
		By assumption, for each $i\geq 1$, at least one of the following two statements is true: that $s(\gamma)\notin s(V_i)$ or $\gamma\notin\overline{V_i}^\circ$.\\
		
		Let $I$ denote the set of values of $i$ for which $s(\gamma)\notin s(V_i)$.
		For $i\in I$, $g_i$ has compact support $K_i\subseteq V_i$, and so $s(K_i)$ is closed in $G^{(0)}$.
		Since $G^{(0)}$ is regular, there are disjoint open neighbourhoods $V'_i$ and $W'_i$ separating $s(K_i)$ and $\{s(\gamma)\}$.
		Let $W_i=s|_{U}^{-1}(W'_i\cap s(U))$, that is, the lift of $W'_i$ to $U$ under the source map.
		Then $g_i|_{W_i}=0$, since $g_i$ is only non-zero on its support, and the source of $g_i$'s support does not intersect the source of $W_i$.
		Let $W:=\bigcap_{i\in I}W_i$, and note that this is an open neighbourhood of $\gamma$.\\
		
		For the second case, let $J$ denote the set of values of $i$ such that $s(\gamma)\in s(V_i)$ but $\gamma\notin\overline{V_i}^\circ$.
		The set $U'_J:=K^\circ\setminus\left(\bigcup_{i\in J}\overline{V_i}\right)$ is then open, non-empty, and has $\gamma$ as a limit point by Lemma~\ref{lem-weirdTopoResult}.
		
		The set $Z:=U'_J\cap W$ is non-empty since $\gamma\in W\cap\overline{U'_J}$, and $W$ is open so intersects $U'_J$.
		Moreover, $Z$ is open as both $U'_J$ and $W$ are, and contained in $K^\circ$ since $U'_J$ is.
		Thus each $g_i$ is zero on $Z$ as $g_i$ is zero on either $W_i$ (if $i\in I$) or is zero on $U'_J$ (if $i\in J$), since $U_J'$ lies in the complement of each $V_i$, which in turn contains the support of $g_i$.
		Thus the sum of the $g_i$ is again zero on $Z$, and so
		$$\left(f-\sum_{i=1}^ng_i\right)\bigg|_Z=f|_Z=1,$$
		as $Z\subseteq K$.
		By Lemma~\ref{lem-bisectionSupBound} we see that $||f-\sum_ig_i||\geq1$ and so $f$ does not lie in the closure of the span of $\Cc_c(V)$ for bisections $V\subseteq H$.
	\end{proof}

	\begin{remark}
		The criterion presented in Proposition~\ref{prop-essGrpdCstarAlgEqualCond} is only non-trivial in the case where \(G\) is both non-Hausdorff and non-effective.

		If \(G\) is effective, Lemma~\ref{lem-effectiveGrpdBisectionRecoveryProperty} reduces the condition in Proposition~\ref{prop-essGrpdCstarAlgEqualCond} to \(H=G\).

		If \(G\) is Hausdorff (and not necessarily effective) then the reduced and essential \(C^*\)-algebras coincide and the inclusion \(C^*_\red(H)\subseteq C^*_\red(G)\) coming from an open subgroupoid \(H\subseteq G\) is exactly that extending the inclusion \(C_c(H)\subseteq C_c(G)\).
		If \(C^*_\red(H)=C^*_\red(G)\) then for any \(\gamma\in G\), a function \(f\in C_c(G)\subseteq C^*_\red(G)\) with \(f(\gamma)=1\) can be approximated by functions  \(f_n\in C_c(H)\).
		We can then eventually deduce that \(f_n(\gamma)\neq 0\) for some \(f_n\), whereby \(\gamma\) belongs to \(H\) and we have \(H=G\).
	\end{remark}
	
	\begin{example}\label{eg-grpdsWithSameEssCstarAlg}
		Consider the line with two origins $G=(0,1]\cup\{0_1,0_0\}$.
		Define the range and source maps $r,s:G\to (0,1]\cup\{0_0\}$ by $r(t)=s(t)=t$ for $t>0$ and $r(0_i)=s(0_i)=0_0$ for $i=0,1$.
		Define a multiplication $G{}_s\times_\red G\to G$ by $t\cdot t=t$ for $t>0$ and $0_i\cdot 0_j=0_{i+j\mod 2}$.
		Then the unit space of $G$ is homeomorphic to the interval $[0,1]$, and the groupoid $G$ satisfies the conditions of Proposition~\ref{prop-essGrpdCstarAlgEqualCond} with the dense open subgroupoid $H=G^{(0)}$.
		Thus $C^*_\ess(G)\cong C[0,1]$, but clearly $G$ is not isomorphic to $[0,1]$ as a groupoid, since it is both non-Hausdorff and contains non-trivial isotropy (although, very little such isotropy).
	\end{example}
	
	Example~\ref{eg-grpdsWithSameEssCstarAlg} illustrates that sometimes removing `small' amounts of isotropy does not change the resulting essential groupoid $C^*$-algebra.
	
	\subsection{More general regular masa inclusions}
	
	Suppose now that \(A\subseteq B\) is a regular masa inclusion (and not necessarily an essential Cartan pair).
	We may still construct the Weyl groupoid and twist as in Section~2.
	Corollary~\ref{cor-mapDescendsToCc(Gtwist)} gives a \({}^*\)-homomorphism \(\Psi_0:\Cc_c(G,\Sigma)\to B\) that intertwines restriction to the unit space with the (necessarily unique) pseudoexpectation associated to the inclusion \(A\subseteq B\).
	Hence, by \cite[Proposition~4.14]{KwasMeyer_EssCrossProd}, the canonical \({}^*\)-homomorphism \(C^*(G,\Sigma)\twoheadrightarrow C^*_\ess(G,\Sigma)\) factors through the canonical map \(C^*(G,\Sigma)\twoheadrightarrow B\), that is, \(B\) is an \emph{exotic \(C^*\)-algebra} for the twist \((G,\Sigma)\).
	
	The inclusion \(A\subseteq B\) is aperiodic by Corollary~\ref{cor-masaIncIsAper}, hence by \cite[Theorem~5.17]{KwasMeyer_EssCrossProd} there is a unique maximal ideal of \(B\) that has trivial intersection with \(A\), named the \emph{hidden ideal}.
	
	\begin{corollary}\label{cor-regMasaIsExoticTwistAlg}
		Let \(A\subseteq B\) be a regular masa inclusion, and let \(I\triangleleft B\) be the hidden ideal, that is, \(I\) is the unique maximal ideal of \(B\) satisfying \(I\cap B=\{0\}\).
		Then the image of \(A\) in the quotient \(B/I\) is masa and detects ideals, and hence \(A\subseteq B/I\) is an essential Cartan inclusion.
		\begin{proof}
			We identify \(A\) with \(C_0(G^{(0)})\) and \(B\) with \(C^*_{\mathrm x}(G,\Sigma)\) as above.
			The maximal ideal \(I\) that is not detected by \(C_0(G^{(0)})\) is then exactly the kernel of the quotient map \(C^*_{\mathrm x}(G,\Sigma)\to C^*_\ess(G,\Sigma)\) by \cite[Theorem~5.28]{KwasMeyer_EssCrossProd}.
			The inclusion \(C_0(G^{(0)})\subseteq C^*_\ess(G,\Sigma)\) is masa by Proposition~\ref{prop-effectiveGrpdGivesMasa}, and the canonical pseudoexpectation for this inclusion is faithful by \cite[Theorem~4.11]{KwasMeyer_EssCrossProd}.
		\end{proof}
	\end{corollary}

	A natural converse question to the statement of Proposition~\ref{prop-effectiveGrpdGivesMasa} is the following: given a twist \((H,\Omega)\) over an \'etale groupoid \(H\) where the inclusion \(A=C_0(H^{(0)})\subseteq B=C^*_{\mathrm x}(H,\Omega)\) is masa for some exotic groupoid \(C^*\)-algebra of \((H,\Omega)\), what properties does \(H\) have?
	It is certainly not possible to assert that \(H\) is effective, as Example~\ref{eg-grpdsWithSameEssCstarAlg} demonstrates.
	The following is not an exhaustive answer to this question, but at least provides a piece of it.
	
	\begin{proposition}\label{prop-regMasaOFAlmostEffective}
		Let \((H,\Omega)\) be a twist over an \'etale groupoid \(H\) with locally compact Hausdorff unit space.
		If the inclusion \(C_0(H^{(0)})\subseteq C^*_{\mathrm x}(H,\Omega)\) is masa for some exotic \(C^*\)-algebra of \((H,\Omega)\) then the unit space \(H^{(0)}\) is dense in the interior of the isotropy of \(H\).
		\begin{proof}
			Suppose that \(A:=C_0(H^{(0)})\subseteq B:=C^*_{\mathrm x}(H,\Omega)\) is masa.
			By Corollary~\ref{cor-regMasaIsExoticTwistAlg}, the quotient \(C^*_{\mathrm x}(H,\Omega)/I\) by the hidden ideal of the inclusion \(C_0(H^{(0)})\subseteq C^*_{\mathrm x}(H,\Omega)\) gives rise to an essential Cartan pair.
			Since the inclusion \(C_0(H^{(0)})\subseteq C^*_{\mathrm x}(H,\Omega)\) is masa, hence aperiodic by Corollary~\ref{cor-masaIncIsAper}, the quotient by the hidden ideal \(I\) corresponds to the essential \(C^*\)-algebra \(C^*_\ess(H,\Omega)\) of \((H,\Omega)\).
			There is then a surjective homomorphism of twists \((H,\Omega)\to (G(A,B/I),\Sigma(A,B/I))\) by Theorem~\ref{thm-groupoidSurjectionProperty}, and the unit space \(H^{(0)}\) is dense in the kernel of the map \(H\to G(A,B/I)\).
			In particular, since \(G\) is effective, the kernel of this map contains the interior of the isotropy of \(H\).
		\end{proof}
	\end{proposition}
	
	\section{Automorphisms of twists and Cartan automorphisms}
	
	Having established the link between twists over certain \'etale groupoids and essential Cartan pairs, one may ask whether automorphisms of one object induce automorphisms of the other.
	We provide constructions in both directions: Cartan automorphisms from twisted groupoid automorphisms and vice versa.
	We also show that these constructions are mutually inverse if the twist is over an effective groupoid, giving an isomorphism of the automorphism group of an essential Cartan pair $A\subseteq B$ to the automorphism group of the corresponding Weyl pair $(G(A,B),\Sigma(A,B))$.
	In particular we show that all automorphisms of twists of effective groupoids that induce essential Cartan pairs come from automorphisms of the induced essential Cartan pair and vice versa.
	
	Throughout this section we assume $(G,\Sigma)$ is a twisted groupoid giving rise to an essential Cartan pair $C_0(G^{(0)})\subseteq C^*_\ess(G,\Sigma)$.
	
	\begin{definition}\label{defn-CartanMorphism}
		Let $(A_1,B_1)$ and $(A_2,B_2)$ be essential Cartan inclusions with local conditional expectations $E_i:B_i\to \Mloc(A_i)$ for $i=1,2$.
		A \emph{Cartan homomorphism} or \emph{homomorphism of Cartan pairs} $(A_1,B_1)\to(A_2,B_2)$ is a ${}^*$-homomorphism $\varphi:B_1\to B_2$ such that
		\begin{enumerate}
			\item\label{ax-normToNorm} $\varphi(N(A_1,B_1))\subseteq N(A_2,B_2)$; and
			\item $E_2\circ\varphi(b)=\varphi\circ E_1(b)$ for all $b\in B_1$ with $E_1(b)\in A_1$.
		\end{enumerate}
		An \emph{isomorphism of Cartan pairs} or \emph{Cartan isomorphism} is a Cartan morphism that is a ${}^*$-isomorphism.
	\end{definition}

	One may expect that an appropriate definition of morphisms of inclusions of \(C^*\)-algebras should preserve the inclusion itself.
	While this is not directly required in the definition of a Cartan homomorphism, it does follow from condition (\ref{ax-normToNorm}) of Definition~\ref{defn-CartanMorphism}.
	The author thanks an anonymous reviewer for pointing this out and providing the following short proof.
	
	\begin{lemma}
		Let \(A_1\subseteq B_1\) and \(A_2\subseteq B_2\) be regular non-degenerate inclusions, and let \(\varphi:B_1\to B_2\) be a \({}^*\)-homomorphism with \(\varphi(N(A_1,B_1))\subseteq N(A_2,B_2)\).
		Then \(\varphi(A_1)\subseteq A_2\).
		\begin{proof}
			Since \(A_1\) is spanned by its positive elements, it suffices to show that positive elements of \(A_1\) are mapped into \(A_2\).
			Fix \(a\in A_1\) positive and let \((u_\lambda)\subseteq A_2\) be an approximate identity.
			Since \(a\) is positive, its square root belongs to \(A_1\), and \(A_1\) is of course contained in the normalisers \(N(A_1,B_1)\).
			Hence we have \(\varphi(a^{\frac12})^*u_\lambda\varphi(a^{\frac12})\in A_2\) for all \(\lambda\), and taking the limit gives \(\varphi(a^{\frac12})^*\varphi(a^{\frac12})=\varphi(a)\in A_2\).
		\end{proof}
	\end{lemma}
	
	If $\varphi:(A_1,B_1)\to(A_2,B_2)$ is an isomorphism of essential Cartan pairs as above then $\varphi(A_1)=A_2$ will hold since $A_1$ is masa in $B_1$.
	That is, $\varphi$ restricts to an isomorphism $A_1\cong A_2$.
	This also allows us to see that \(\varphi^{-1}\) fulfils the connditions of Definition~\ref{defn-CartanMorphism}, and so the inverse isomorphism of a Cartan isomorphism is also a Cartan isomorphism, justifying the terminology.
	
	\begin{lemma}\label{lem-CartanMorphCondForNormalisers}
		Let $(A_1,B_1)$ and $(A_2,B_2)$ be essential Cartan inclusions with local conditional expectations $E_i:B_i\to \Mloc(A_i)$ for $i=1,2$ and let $\varphi:B_1\to B_2$ be a ${}^*$-homomorphism.
		If $\varphi$ intertwines conditional expectations and $\varphi(A_1)= A_2$ then $\varphi$ is a Cartan morphism.
		\begin{proof}
			We need only show that $\varphi$ maps normalisers to normalisers.
			Let $n\in N(A_1,B_1)$ be a normaliser and fix $a\in A_2$.
			Since $\varphi(A_1)=A_2$, there exists $a'\in A_1$ with $\varphi(a')=a$.
			Then $\varphi(n)^*a\varphi(n)=\varphi(n^*a'n)\in\varphi(A_1)=A_2$.
		\end{proof}
	\end{lemma}
	
	If the pairs $(A_1,B_1)$ and $(A_2,B_2)$ are aperiodic inclusions then the conditional expectations are unique by \cite[Theorem~3.6]{KwasMeyer_AperTopoFree}.
	Thus the expectations are intertwined by any $^*$-homomorphism that restricts to an isomorphism of the subalgebras $A_1$ and $A_2$.
	
	\begin{lemma}\label{lem-mapCartAutToGrpdAutIsHom}
		The assignment $(A,B)$ to $(G(A,B),\Sigma(A,B))$ taking an essential Cartan pair to its Weyl groupoid and twist is functorial for Cartan isomorphisms.
		Explicitly, the identity Cartan isomorphism on \((A,B)\) induces the identity morphism on the Weyl twist, and if $\varphi$ and $\psi$ are automorphisms of $(A,B)$ then Theorem~\textup{\ref{thm-groupoidSurjectionProperty}} yields isomorphisms of twists \(\tilde\beta_\varphi\) and \(\tilde\beta_\psi\) satisfying $\tilde\beta_{\varphi\circ\psi}=\tilde\beta_\varphi\circ\tilde\beta_\psi$.
		\begin{proof}
			For brevity we denote the Weyl twist for \((A,B)\) by \((G,\Sigma)\) and we identify \((A,B)\) with the corresponding pair \((C_0(G^{(0)}),C^*_\ess(G,\Sigma))\) via Theorem~\ref{thm-isoEssTwstGrpdCstarAlg}.
			Theorem~\ref{thm-groupoidSurjectionProperty} ensures the existence of the twisted groupoid homomorphisms \(\tilde\beta_\varphi\) and \(\tilde\beta_\psi\), and Corollary~\ref{cor-WeylGrpdUniqueAmongEffectiveGrpds} ensures it is an isomorphism.
			It is clear that the identity Cartan morphism on \((A,B)\) is mapped ot the identity map on \((G,\Sigma)\).
			We lastly observe 
			\[\tilde\beta_\varphi\circ\tilde\beta_\psi[n,x]=[\varphi(\psi(n)),((\varphi\circ\psi)|_{C_0(G^{(0)})}^*)^{-1}(x)]=\tilde\beta_{\varphi\circ\psi}[n,x],\]
			completing the proof.
		\end{proof}
	\end{lemma}
	
	\begin{proposition}\label{prop-CartanAutFromTwistedGrpdAut}
		Let \((G,\Sigma)\) be a twist over an essential \'etale groupoid, so that the pair \((C_0(G^{(0)}),C^*_\ess(G,\Sigma))\) is an essential Cartan pair.
		Let $(\beta,\tilde\beta):(G,\Sigma)\to (G,\Sigma)$ be an automorphism of twists.
		Then there is an automorphism of Cartan pairs
		\[\Phi_{(\beta,\tilde\beta)}:(C_0(G^{(0)}),C^*_\ess(G,\Sigma))\to (C_0(G^{(0)}),C^*_\ess(G,\Sigma))\] 
		mapping a section $f\in\Cc_c(G,\Sigma)$ to $f\circ\beta^{-1}$.
		\begin{proof}
			Since $(\beta,\tilde\beta)$ is a twisted automorphism, so is $(\beta^{-1},\tilde\beta^{-1})$, and in particular $\beta^{-1}$ is a homeomorphism mapping bisections to bisections, so precomposing sections with \(\beta^{-1}\) maps sums of compactly supported sections to sums of compactly supported sections.
			Let $\Phi_{(\beta,\tilde\beta)}^0:\Cc_c(G,\Sigma)\to\Cc_c(G,\Sigma)$ be the map $\Phi_{(\beta,\tilde\beta)}^0(f)=f\circ\beta^{-1}$.
			That $\Phi_{(\beta,\tilde\beta)}^0$ is a ${}^*$-isomorphism follows from the fact that $(\beta,\tilde\beta)$ is an isomorphism of twists.
			
			The isomorphism $\Phi_{(\beta,\tilde\beta)}^0$ then extends to an automorphism $\Phi_{(\beta,\tilde\beta)}$ of the full twisted groupoid $C^*$-algebra $C^*(G,\Sigma)$.
			Let $R$ be the restriction map from Proposition~\ref{prop-commutingCondExps} and note that $R$ extends to the canonical conditional expectation $EL:C^*(G,\Sigma)\to\Mloc(C_0(G^{(0)}))$.
			Then clearly $R\circ\Phi_{(\beta,\tilde\beta)}^0=\Phi_{(\beta,\tilde\beta)}^0\circ R$, so $\Phi_{(\beta,\tilde\beta)}$ descends to an automorphism of the quotient $C^*_\ess(G,\Sigma)$.
			
			To see that $\Phi_{(\beta,\tilde\beta)}$ is an automorphism of essential Cartan pairs, note that $\beta$ restricts to a homeomorphism of the unit space $G^{(0)}$, and so for $f\in C_0(G^{(0)})$ we have $\Phi_{(\beta,\tilde\beta)}(f)=f\circ\beta^{-1}$ which belongs to $C_0(G^{(0)})$.
			That $\Phi_{(\beta,\tilde\beta)}$ maps normalisers to normalisers follows from Lemma~\ref{lem-CartanMorphCondForNormalisers}.
		\end{proof}
	\end{proposition}
	
	\begin{corollary}\label{cor-functorForIsos}
		The assignment $(G,\Sigma)$ to $(C_0(G^{(0)}),C^*_\ess(G,\Sigma))$ is functorial, lifting isomorphisms of twists to Cartan isomorphisms.
		That is, the map ${(\beta,\tilde\beta)}\mapsto\Phi_{(\beta,\tilde\beta)}$ from automorphisms of the twist $(G,\Sigma)$ to Cartan automorphisms of $(C_0(G^{(0)}),C^*_\ess(G,\Sigma))$ is a group homomorphism.
		\begin{proof}
			Let ${(\beta,\tilde\beta)},{(\beta',\tilde\beta')}\in\Aut(G,\Sigma)$.
			For each slice $U\subseteq G$ and $f\in\Cc_c(U,\Sigma)$ we have $\Phi_{(\beta,\tilde\beta)}(\Phi_{(\beta',\tilde\beta')}(f))=f\circ(\beta'^{-1}\circ\beta)=\Phi_{(\beta\circ\beta',\tilde\beta\circ\tilde\beta')}(f)$.
			Such normalisers span a dense subalgebra of $C^*_\ess(G,\Sigma)$ and both $\Phi_{(\beta,\tilde\beta)}\circ\Phi_{(\beta',\tilde\beta')}$ and $\Phi_{(\beta\circ\beta',\tilde\beta\circ\tilde\beta')}$ are linear and isometric so the equality extends to all elements of $C^*_\ess(G,\Sigma)$.
		\end{proof}
	\end{corollary}
	
	\begin{proposition}\label{prop-grpdAutIsTrivOnAlgStuff}
		Let \((G,\Sigma)\) be a twist over an effective groupoid.
		Let $(\beta,\tilde\beta)$ be an automorphism of the twist $(G,\Sigma)$.
		If the induced Cartan automorphism $\Phi_{(\beta,\tilde\beta)}$ is equal to the identity on $C^*_\ess(G,\Sigma)$ then there is a dense subset $D\subseteq G$ with $\beta|_D=\Id_D$.
		Moreover, if $\beta\neq\Id_G$ then there exists $\eta\in \overline{G^{(0)}}^\circ\setminus G^{(0)}$ and \(G\) is not effective.
		\begin{proof}
			Suppose that $U\subseteq G$ is an open bisection with $\beta(\gamma)\neq\gamma$ for all $\gamma\in U$.
			We claim that there is a non-empty open subset $V\subseteq U$ with $\beta^{-1}(V)\cap V=\emptyset$.
			If $\beta^{-1}(U)\cap U=\emptyset$ then the choice of $V:=U$ suffices.
			Else, if $\beta^{-1}(U)\cap U\neq\emptyset$ then there exists some $\gamma\in U$ with $\beta^{-1}(\gamma)\in U$.
			Since $U$ is Hausdorff, we can pick open $V_1,V_2\subseteq U$ with $\gamma\in V_1$, $\beta^{-1}(\gamma)\in V_2$, and $V_1\cap V_2=\emptyset$.
			Set $V:=V_1\cap \beta(V_2)$.
			Then $$V\cap\beta^{-1}(V)=V_1\cap\beta(V_2)\cap\beta^{-1}(V_1)\cap V_2\subseteq V_1\cap V_2=\emptyset$$ 
			as required.
			In particular, since $V$ and $\beta(V)$ belong to the same bisection and are disjoint the product slices $V^{-1}\beta^{-1}(V)$ and $\beta^{-1}(V)^{-1}V$ are empty.
			Thus for any $f\in\Cc_c(V,\Sigma)$ the compositions $f^*\Phi_{(\beta,\tilde\beta)}(f)$ and $\Phi_{(\beta,\tilde\beta)}(f)^*f$ are zero as their supports are $V^{-1}\tilde\beta(V)$ and $\beta^{-1}(V)^{-1}V$, respectively.
			Since $\Phi_{(\beta,\tilde\beta)}$ is the identity on $C^*_\ess(G,\Sigma)$, we then have $0=||f-\Phi_{(\beta,\tilde\beta)}(f)||^2=||f^*f-f^*\Phi_{(\beta,\tilde\beta)}(f)-\Phi_{(\beta,\tilde\beta)}(f)^*f+\Phi_{(\beta,\tilde\beta)}(f^*f)||=||f^*f+\Phi_{(\beta,\tilde\beta)}(f^*f)||$, and so $f^*f=-\Phi_{(\beta,\tilde\beta)}(f^*f)$.
			But both $f^*f$ and $\Phi_{(\beta,\tilde\beta)}(f^*f)$ are positive elements of $C^*_\ess(G,\Sigma)$, so $f^*f=0$ must hold.
			Thus $f=0$ for all $f\in\Cc_c(V,\Sigma)$, whereby $V=\emptyset$, which is a contradiction.
			Hence there are no open bisections $U\subseteq G$ with $\beta(\gamma)\neq\gamma$ for all $\gamma\in U$, and so the subset $D$ of $G$ on which $\beta$ acts trivially is dense in $G$.\\
			
			Suppose $\beta\neq\Id_G$.
			Then there exists $\gamma\in G$ with $\beta(\gamma)\neq\gamma$.
			Let $U\subseteq G$ be an open bisection with $\gamma\in U$.
			Then $U\cap D$ is dense in $U$, and so $(U\cap D)^{-1}\beta(U\cap D)=(U\cap D)^{-1}(U\cap D)=s(U\cap D)$ is dense in the bisection $U^{-1}\beta(U)$.
			In particular $U^{-1}\beta(U)$ is an open bisection contained in the closure of the unit space $G^{(0)}$, so in particular $U^{-1}\beta(U)\subseteq \overline{G^{(0)}}^\circ$.
			Since $\beta(\gamma)\neq\gamma$, we have that $\eta:=\gamma^{-1}\beta(\gamma)$ is not a unit, and so $\eta\in\overline{G^{(0)}}^\circ\setminus G^{(0)}$.
		\end{proof}
	\end{proposition}
	
	\begin{corollary}
		Let \((G,\Sigma)\) be a twist over an effective groupoid.
		An automorphism ${(\beta,\tilde\beta)}$ of $(G,\Sigma)$ gives rise to the identity automorphism $\Phi_{(\beta,\tilde\beta)}$ of $C^*_\ess(G,\Sigma)$ if and only if ${(\beta,\tilde\beta)}=(\Id_G,\Id_\Sigma)$.
		\begin{proof}
			If $\tilde\beta$ is not the identity but $\Phi_{(\beta,\tilde\beta)}=\Id_{C^*_\ess(G,\Sigma)}$ then Proposition~\ref{prop-grpdAutIsTrivOnAlgStuff} shows that $\overline{G^{(0)}}^\circ$ is strictly larger than $G^{(0)}$, leading to a contradiction.
		\end{proof}
	\end{corollary}
	
	\begin{corollary}\label{cor-twistGrpdAutsToCartInjIfEffective}
		If $G$ is effective then the homomorphism 
		$$\Phi:\Aut(G,\Sigma)\to\Aut(C_0(G^{(0)}),C^*_\ess(G,\Sigma))$$
		is injective.
	\end{corollary}
	
	\begin{theorem}\label{thm-isoOfAutGrps}
		For the Weyl twist $(G,\Sigma)$ of an essential Cartan pair $(A,B)$ the constructions in Theorem~\textup{\ref{thm-groupoidSurjectionProperty}} and Proposition~\textup{\ref{prop-CartanAutFromTwistedGrpdAut}} are mutually inverse.
		That is, for a twisted groupoid automorphism ${(\beta,\tilde\beta)}:(G,\Sigma)\to(G,\Sigma)$ we have $\beta_{\Phi_{(\beta,\tilde\beta)}}=\beta$ and \(\tilde\beta_{\Phi_{(\beta,\tilde\beta)}}=\tilde\beta\).
		For a Cartan automorphism $\Phi:(C_0(G^{(0)}),C^*_\ess(G,\Sigma))\to(C_0(G^{(0)}),C^*_\ess(G,\Sigma))$ we have $\Phi_{(\beta_\Phi,\tilde\beta_\Phi)}=\Phi$.
		Hence $\Aut(G,\Sigma)$ and $\Aut(C_0(G^{(0)}),C^*_\ess(G,\Sigma))$ are isomorphic as groups via these constructions.
		\begin{proof}
			Let $\Phi:(C_0(G^{(0)}),C^*_\ess(G,\Sigma))\to(C_0(G^{(0)}),C^*_\ess(G,\Sigma))$ be a Cartan automorphism.
			By Lemma~\ref{lem-mapCartAutToGrpdAutIsHom} we have $\tilde\beta_\Phi[n,x]=[\Phi(n),\theta_\Phi(x)]$, and so $\tilde\beta_\Phi^{-1}[n,x]=[\Phi^{-1}(n),\theta_\Phi^{-1}(x)]$.
			For a normaliser $n\in N(C_0(G^{(0)}),C^*_\ess(G,\Sigma))$ let $U_n:=\{[\alpha_n,x]:x\in\dom(n)\}$ be the bisection specified by $n$, and fix $f\in\Cc_c(U_n,\Sigma)$.
			Then $f=\hat{n}h$ for some $h\in C_0(G^{(0)})$ by Lemma~\ref{lem-compactlySupportedSectionsAsEvalutatedNormalisers}, and moreover for all $[n,x]$ in the open support of $f$ we have $[n,x]=[znh,x]$ for some fixed $z\in\TT$ (namely $z=\overline{h(x)}$).
			Thus
			\begin{align*}
				\Phi(f)[\Phi(f),\theta_\Phi(x)]&=\Phi(\hat{n}h)[\Phi(\hat{n}h),\theta_\Phi(x)]\\
				&=\sqrt{\Phi((nh)^*(nh))(\theta_\Phi(x))}\\
				&=\sqrt{(nh)^*(nh)(x)}\\
				&=\sqrt{f^*f(x)}\\
				&=f[f,x],
			\end{align*}
			and
			\begin{align*}
				\Phi_{(\beta_\Phi,\tilde\beta_\Phi)}(f)[\Phi(f),\theta_\Phi(x)]&=f(\tilde\beta_\Phi^{-1}[\Phi(f),\theta_\Phi(x)])\\
				&=f[\Phi^{-1}\Phi(f),\theta_\Phi^{-1}\theta_\Phi(x)]\\
				&=f[f,x].
			\end{align*}
			Hence $\Phi=\Phi_{(\beta_\Phi,\tilde\beta_\Phi)}$, so the map $\Aut(G,\Sigma)\to\Aut(C_0(G^{(0)}),C^*_\ess(G,\Sigma))$ is surjective.
			The Weyl groupoid is effective by Lemma~\ref{lem-weylGrpdIsEffective}, and so Corollary~\ref{cor-twistGrpdAutsToCartInjIfEffective} gives that this map $\Aut(G,\Sigma)\to\Aut(C_0(G^{(0)}),C^*_\ess(G,\Sigma))$ is injective, hence an isomorphism.
			The above calculation shows that the construction in Theorem~\ref{thm-groupoidSurjectionProperty} is a one-sided inverse to this isomorphism, hence is the two-sided inverse and is itself an isomorphism.
		\end{proof}
	\end{theorem}

	The isomorphism of Theorem~\ref{thm-isoOfAutGrps} generalises a similar result of Raad \cite{Raad_GeneraliseRenault}, where the isomorphism is found in the case of Renault's Cartan pairs and reduced twisted groupoid $C^*$-algebras.
	
	In a recent preprint, Komura \cite[Propsition~2.2.2]{Komura_HomsGrpdAlgs} characterises the Cartan automorphism group of \(C^*_\red(G)\) for Hausdorff \'etale effective groupoids \(G\) as the crossed product of the automorphism group of \(G\) by the group of continuous homomorphisms \(G\to\TT\).
	Komura's proof does not explicitly use the Cartan structure of the inclusion \(C_0(G^{(0)})\subseteq C^*_\red(G)\), but this inclusion is always Cartan in the sense of Renault by \cite[Propositions~4.2,4.3]{Renault_Cartan}.
	
	\begin{corollary}
		The category of twists of effective \'etale groupoids with locally compact Hausdorff unit spaces with isomorphisms of twists as morphisms is equivalent to the category of essential Cartan pairs with Cartan isomorphisms via the functor taking a twist \((G,\Sigma)\) to the pair \((C_0(G^{(0)}),C^*_\ess(G,\Sigma))\).
		\begin{proof}
			Corollary~\ref{cor-functorForIsos} gives that this assignment is a functor, Theorem~\ref{thm-isoOfAutGrps} gives that it is full and faithful.
			Proposition~\ref{prop-effectiveGrpdGivesMasa} implies that effective groupoids yield masa inclusions, hence essential Cartan inclusions, and Theorem~\ref{thm-isoEssTwstGrpdCstarAlg} implies that all essential Cartan inclusions are of this form.
			Hence the functor is essentially surjective.
		\end{proof}
	\end{corollary}

	\printbibliography

\end{document}